\documentclass[11pt,a4paper]{amsart}

\pdfoutput=1

\newcommand{\figdir}{./}
\newcommand{\figdirNew}{./}

\usepackage[foot]{amsaddr}

\usepackage[margin=3cm]{geometry}
\usepackage{amsmath,amsfonts,amssymb,amsthm}
\usepackage{graphicx,enumerate}

\newcommand{\pdpd}[2]{\frac{\partial #1}{\partial #2}}

\newcommand{\supp}{\operatorname{supp}}

\newcommand{\RR}{{\mathbb{R}}}
\newcommand{\ZZ}{{\mathbb{Z}}}

\newcommand{\bs}{\backslash}
\newcommand{\FourierF}{\mathcal{F}}

\newcommand{\sgn}{\operatorname{sgn}}
\newcommand{\PV}{\operatorname{p.v.}}

\newtheorem{definition}{Definition}
\newtheorem{theorem}{Theorem}
\newtheorem{proposition}[theorem]{Proposition}

\title[A dispersion minimizing scheme for the 3-D Helmholtz equation]
{A dispersion minimizing scheme for the 3-D Helmholtz equation based
  on ray theory}

\author{Christiaan C. Stolk}
\address{Korteweg-de Vries Institute for Mathematics, Science
  Park 105-107, 1090 XG Amsterdam The Netherlands}
\email{C.C.Stolk@uva.nl}

\begin{document}
\begin{abstract}
We develop a new dispersion minimizing compact finite difference
scheme for the Helmholtz equation in 2 and 3 dimensions. The scheme is
based on a newly developed ray theory for difference equations. A
discrete Helmholtz operator and a discrete operator to be applied to
the source and the wavefields are constructed. Their coefficients are
piecewise polynomial functions of $hk$, chosen such that phase and
amplitude errors are minimal. The phase errors of the scheme are very
small, approximately as small as those of the 2-D quasi-stabilized FEM
method and substantially smaller than those of alternatives in 3-D,
assuming the same number of gridpoints per wavelength is used. In
numerical experiments, accurate solutions are obtained in constant and
smoothly varying media using meshes with only five to six points per
wavelength and wave propagation over hundreds of wavelengths. When
used as a coarse level discretization in a multigrid method the scheme
can even be used with downto three points per wavelength.  Tests on
3-D examples with up to $10^8$ degrees of freedom show that with a 
recently developed hybrid solver, the use of coarser meshes can lead 
to corresponding savings in computation time, resulting in good 
simulation times compared to the literature.
\end{abstract}

\maketitle

\section{Introduction}

We consider the discretization on regular meshes of the Helmholtz
equation
\begin{equation} \label{eq:Helmholtz_equation}
  -\Delta u - k(x)^2 u = f 
\end{equation}
with large and variable $k$. These methods are widely used for
simulations on unbounded domains, for example in exploration
geophysics, using domain sizes, in three dimensions, of up to hundreds
of wavelengths \cite{RiyantiEtAl2007,CalandraEtAl2013,PoulsonEtAl2013,
Stolk2014Preprint}.

A key issue for such discretizations are the dispersion (phase)
errors, that are closely related to pollution errors
\cite{BabuskaEtAl1995}. Typically, the propagating wave solutions to
the discrete and continuous equations have slightly different
wavelengths. These wavelength errors are also referred to as phase
velocity or phase slowness errors, in which case they are differently
normalized. They lead to phase errors in the solution that grow with
the distance from the source. A second important consideration is
solver cost.  The discretized Helmholtz operator should of course be
cheap to apply and/or invert.

A class of discretizations, that performs relatively
well on these criteria, is given by so called compact finite
difference methods, that use a $3 \times 3$ square or $3 \times 3
\times 3$ cubic stencil in two resp.\ three dimensions. The
corresponding discrete Helmholtz operators can be efficiently applied
and inverted compared for example to standard finite difference or
finite element methods. Many authors have studied such discretizations
and obtained formulae for the coefficients as a function of $k$ and
the grid spacing $h$
\cite{BabuskaEtAl1995,HarariTurkel1995,JoShinSuh1996,OpertoEtAl2007,
  Sutmann2007,ChenChengWu2012,TurkelEtAl2013,StolkEtAl2014}. We will discuss these
schemes more in detail below, and compare their phase slowness errors
with those of standard finite differences and Lagrange finite element
methods on regular meshes.

To design such methods, several strategies have been followed.  One
approach is too construct schemes of higher order, for example order four
order six, see \cite{HarariTurkel1995,Sutmann2007,TurkelEtAl2013} and 
the references in \cite{Sutmann2007}. Another approach is to stay
with second order schemes but minimize the dispersion errors, because
these are the dominant errors for long distance wave propagation
\cite{BabuskaEtAl1995,JoShinSuh1996,
  OpertoEtAl2007,ChenChengWu2012,StolkEtAl2014}. From the point of
view of phase slowness errors, the sixth order schemes of
\cite{Sutmann2007} and
\cite{TurkelEtAl2013} and the quasi-stabilized FEM (QS-FEM) scheme of
\cite{BabuskaEtAl1995} are the best, see the results below. The latter
has the smallest phase slowness errors by a substantial margin, but is
only available in 2-D.

Alternative methods include higher order finite elements. An advantage
of these methods is the better theory for the behavior of the errors
in the limit that the grid spacing goes to zero, see for
example \cite{Ihlenburg1998,MelenkSauter2011}.

In this paper we introduce a new second order dispersion minimizing
scheme in 2 and 3 dimensions with phase slowness errors comparable to
those of QS-FEM. Accurate amplitudes are obtained as well using new
amplitude correction operators. A theoretical justification is given
using a newly developed ray theory for Helmholtz-like difference
equations.  This theory is remarkably similar to the continuous
theory, when both are formulated in terms of the symbols associated
with the operators.  With numerical examples we show the potential for
accurate and fast simulation on relatively coarse meshes. In addition
we show applications where the method is used as a coarse level
discretization in multigrid solvers

We will briefly describe the methodology and the results.  It is known
that the second order, compact finite difference discretizations of
the Helmholtz operator form a 3 or 5 parameter family, in 2 and 3
dimensions respectively, and that by choosing parameters in a certain
way, the phase slowness errors can be reduced compared to standard
schemes \cite{JoShinSuh1996, OpertoEtAl2007}. When coefficients are
allowed to depend on $hk$ in a piecewise constant
\cite{ChenChengWu2012} or piecewise linear fashion
\cite{StolkEtAl2014}, they can be further reduced. In this paper we
let the parameters depend in a $C^1$ fashion on $hk$ through third
order Hermite interpolation and obtain a further reduction of the phase 
slowness errors.

Dispersion minimizing schemes are typically intended for use on quite
coarse meshes, and a theoretical understanding that does not involve
the limit $h k \rightarrow 0$ is therefore of considerable
interest. For this reason we consider ray theory for Helmholtz-like
difference equations.

Ray theory for continuous Helmholtz equations is well known
\cite{Duistermaat1996}.
Solutions are sought in the form $A(x) e^{i \omega \Phi(x)}$. If $k$
is smooth, $\Phi$ satisfies a certain eikonal equation and $A$ a
certain transport equation, than such solutions approximate the true
solutions increasingly well in the limit $\omega \rightarrow \infty$.
Here we develop a similar theory for Helmholtz-like difference
equations. We can then choose the discrete scheme such that the phase
and amplitude functions associated with the discrete operator
approximate match those of the continuous operator well.  As can be
expected, schemes with small phase slowness errors have accurate phase
functions. By introducing amplitude correction operators, accurate
amplitude of the ray-theoretic solutions are obtained.

We are interested in two ways of applying the discretized Helmholtz
operators. The first is simply as a discretization of
(\ref{eq:Helmholtz_equation}), where the criterion is that the
discrete solutions should approximate the true solutions well. Here we
are particularly interested in the use of coarse meshes, say downto
five or six points per wavelength, which are for example applied in
exploration geophysics
\cite{JoShinSuh1996,OpertoEtAl2007,KnibbeEtAl2014}.  The second
application is internally in multigrid based solvers.  In a multigrid
method, the original mesh is coarsened by a factor two one or more
times. On each of the new meshes a discretization of the operator is
required. In this application the main criterion for a good
discretization is that the multigrid method converges rapidly. The
results concerning the application in multigrid methods are also of
interest for recently developed two-grid or multigrid methods with
inexact coarse level inverses \cite{CalandraEtAl2013,
  Stolk2014Preprint}, which are currently some of the fastest solvers
in the literature. (The method of \cite{Stolk2014Preprint} will
actually be tested here.) Below we will write sometimes the fine level
mesh for the original, uncoarsened mesh.

The small phase slowness errors for IOFD suggest that accurate solutions
are possible even when quite coarse meshes are used, say downto five
or six points per wavelength. We will show that this is indeed the
case using numerical examples with constant, and smoothly varying
velocity models (recall that $k(x) = \frac{\omega}{c(x)}$ with $c$ the
medium velocity). 

We then consider the application of the IOFD discretization as coarse
level discretization in multigrid based solvers. We will show that in
this case IOFD can be used with very coarse meshes with downto three
points per wavelength. With such meshes, solutions are generally not
accurate enough for direct use, but the approximate solutions can
still be used fruitfully in a multigrid method, where they are refined
and iteratively improved. This is established using a set of
two-dimensional examples, in which a two-grid method with IOFD at both
levels converges rapidly (see also the results discussed in the next
paragraph). As explained in \cite{StolkEtAl2014}, for the good
convergence it is necessary to have very small phase slowness errors at
these very coarse meshes. The IOFD method (in two and three
dimensions) and the QS-FEM method (in two dimensions) are the only
discretizations that have this property to our knowledge, and appear
to be uniquely suitable for this application.

In 3-D, the fact that a coarser mesh is used does not necessarily
imply lower simulation cost. That depends also on the behavior of the
solver. To investigate this aspect we present tests with a recently
developed solver described in \cite{Stolk2014Preprint}. The solver
uses a two-grid method with an inexact coarse level inverse, given by
a double sweep domain decomposition preconditioner. As described in
the previous paragraph, IOFD will also be used as coarse level
discretization. Using the SEG-EAGE Salt Model with up to $10^8$
degrees of freedom as example, we find that for downto six points per
wavelength the cost per degree of freedom changes little when the
frequency is increased. Computation time compare favorably to some of
the results in the literature.

The outline of this work is as follows. In 
section~\ref{sec:theory_constant} the theory for finite difference
discretizations of the Helmholtz equation with constant $k$ is
developed. The symbols and phase slownesses are defined and the
discrete Green's function is studied.
In section~\ref{sec:theory_variable} we consider the case of variable
$k$ and describe ray theory for discrete Helmholtz equations.
In section~\ref{sec:phase_slowness_errors_existing} we compute the 
phase slowness errors of various existing
schemes, as a reference for the new method.  In section~\ref{sec:IOFD}
we introduce our new interpolated optimized finite difference method.
Section~\ref{sec:simulations} contains some numerical simulations
illustrating the accuracy of the solutions when using the IOFD
discretization.  Section~\ref{sec:twogrid} discusses the use of IOFD
in multigrid based solvers. Finally, section~\ref{sec:discussion}
contains a brief discussion of some further aspects.

\section{Theory of discrete Helmholtz equations with constant $k$%
\label{sec:theory_constant}}

In this section we study finite difference discretizations of 
Helmholtz equation
\begin{equation}
  H u = f , \qquad H = - \Delta - k^2 
\end{equation}
in case $k$ is constant. We will assume the grid is given by
$(h\ZZ)^d$. In this and the next section it is convenient to write
$\alpha,\beta, ...$ for multi-indices associated with grid points,
such that with $\alpha=(\alpha_1,\ldots,\alpha_d)$ is associated the
grid point $h \alpha$. A difference operator will be viewed as an
operator on functions of $x \in (h\ZZ)^d$. In this and the next
section the dimension $d$ can be any positive integer.

For constant $k$, a finite difference discretization of the Helmholtz
operator $H$ is a translation invariant difference operator with
coefficients depending on the grid spacing $h$ and on $k$.  By
dimensional analysis we may assume that the matrix elements
$p_{\alpha,\beta}$ of such a difference operator $P$ are defined in terms
of a finite set of functions $f_\gamma$ by
\begin{equation} \label{eq:define_P_constant}
  p_{\alpha,\beta} = \frac{1}{h^2} f_{\alpha - \beta}(h k) , 
\end{equation} 
where $f_\gamma$ is only nonzero for $\gamma$ in some finite set
$\operatorname{Sten}(P)$.

We will first consider the action of such an operator in the Fourier
domain and define the associated symbol and phase slownesses. We next
define a ``dimensionally reduced'' symbol. Then we consider the
discrete Green's function, i.e.\ solutions to the equation
\begin{equation} \label{eq:Greens_functions_general_spatial}
  P u = \delta ,
\end{equation}
where the $\delta$ function on $(h\ZZ)^d$ is defined by $\delta(h
\alpha) = h^{-d} \delta_{\alpha_1,0} \ldots \delta_{\alpha_d,0}$.  We
obtain the general solution of this equation in the
Fourier domain and the asymptotics in the spatial domain,
and determine the same information for the unique outgoing
solutions. 

In the last part of this section we consider a modification of
(\ref{eq:Greens_functions_general_spatial}) where first a function $v$
is determined that satisfies
\begin{equation} \label{eq:Qdelta_source}
  P v = \tilde{Q} \delta
\end{equation}
and then $u$ is set equal to
\begin{equation} \label{eq:Qoperator_to_Qdelta_GF}
  u = \hat{Q} v .
\end{equation}
In this case we assume $\tilde{Q}$ and $\hat{Q}$ are difference
operators of order zero. Based on translation invariance and
dimensional reduction, we assume that their matrix elements
$\tilde{q}_{\alpha,\beta}$ and $\hat{q}_{\alpha,\beta}$ are given by
\begin{equation} \label{eq:define_matrix_elements_q}
  \tilde{q}_{\alpha,\beta} = \tilde{g}_{\alpha - \beta}( h k) , \qquad
  \hat{q}_{\alpha,\beta} = \hat{g}_{\alpha - \beta}( h k) ,
\end{equation}
where the $\tilde{g}_\gamma$, $\hat{g}_\gamma$ are smooth functions
that are only nonzero for $\gamma$ in finite sets 
$\operatorname{Sten}(\tilde{Q})$, $\operatorname{Sten}(\hat{Q})$.
A solution $u$ to such a system will be called a modified Green's
function. Our discretization of the Helmholtz equation will be a system of the
form (\ref{eq:Qdelta_source}) and (\ref{eq:Qoperator_to_Qdelta_GF}),
where $\delta$ is replaced by the right hand side $f$.

\subsection{Symbol and phase slownesses\label{subsec:symbol_phase_slowness}}

To define the symbol, we first define the forward and inverse Fourier
transforms of a function $u(x)$, $x \in (h\ZZ)^d$. They are given by
\begin{align}
  \FourierF u (\xi) = {}& h^d \sum_{x \in (h\ZZ)^d}  u(x) e^{-i \xi \cdot x}
\\
  \label{eq:IFT_integral}
  \FourierF^{-1} U(x) = {}&
    (2\pi)^{-d} \int_{[-\pi/h,\pi/h]^d} U(\xi) e^{i \xi \cdot x} \,
    d\xi ,
\end{align}
where the domain of $\FourierF u$ is $[-\pi/h,\pi/h]^d$.
For constant $k$ the finite difference operator $P$ acts like a
multiplication in the Fourier domain
\begin{equation} 
  \FourierF ( P u )(\xi) = P(\xi) \FourierF u(\xi) .
\end{equation}
where the function $P(\xi)$, called the {\em symbol}, is given by 
\begin{equation}
  P(\xi) = h^{-2} \sum_\gamma f_\gamma(hk) e^{i h\gamma \cdot \xi} .
\end{equation}
This is similar to the continuous case, where the Helmholtz operator 
$H$ acts by multiplication with $H(\xi) = \xi^2 - k^2$ in the
Fourier domain.

The Helmholtz equation has propagating plane wave solutions. These are
functions $u = e^{ix\cdot \xi}$ that satisfy the homogeneous Helmholtz equation
\begin{equation}
  H e^{i \xi\cdot x} = 0 .
\end{equation}
They are exactly the plane waves for which $\xi$ is in the zeroset $Z_H$
of the symbol of $H(\xi) = \xi^2 - k^2$ (This is of course the set of
vectors of length $k$ for $H$ as defined, but the concept applies 
more generally.)
If $P$ is a translation invariant discretization of $-\Delta - k^2$ on
$\RR^d$ we can similarly look for vectors $\xi$ such that
\begin{equation} \label{eq:plane_wave}
  P e^{ix \cdot \xi} = 0 .
\end{equation}
These are the vectors in the zero set $Z_P$ of $P(\xi)$. 

If $P$ is a discretization of the Helmholtz operator $H$ then
typically the set $Z_P$ is close to, but not identical to $Z_H$. In
other words, there are small differences in the wave vectors of
the propagating waves, $Z_P \neq Z_H$. If $Z_P$ and $Z_H$ can be 
parameterized by angle, i.e.\ 
\begin{equation} \label{eq:Z_star_shaped}
  Z_P = \{ g_P(\theta) \theta \,|\, \theta \in S^{d-1} \}  ,
\end{equation}
and similar for $Z_H$ (for $H(\xi) = \xi^2 -k^2$ this is of course the
case), we define the relative wave number error as a function of
$\theta \in S^{d-1}$ by
\begin{equation} \label{eq:define_relative_phase_error}
  \delta_{\rm ph}(\theta) = \frac{g_P(\theta)}{g_H(\theta)} - 1 .
\end{equation}

Closely related quantities are the phase slowness and phase velocity
errors. If $k = \frac{\omega}{c(x)}$, $\xi \in Z_P$, there are
associated phase slowness $s_{\rm ph}$ and phase velocity vectors
$v_{\rm ph}$ given by
\begin{equation}
  s_{\rm ph} = \omega^{-1} \xi , \qquad
  v_{\rm ph} = \frac{\omega \xi}{\| \xi \|^2} ,
\end{equation}
see e.g.\ \cite{Cohen2002}. The quantity $\delta_{\rm ph}$ defined in 
(\ref{eq:define_relative_phase_error}) may hence also be called the
relative phase slowness error, or simply phase slowness error.

The actual phase error between a numerical and an exact solution is
given by (see also subsection~\ref{subsec:discrete_Greensfunction_asymptotics})
\begin{equation} \label{eq:phase_error_from_dispersion_error}
  \text{phase error} = 2\pi \delta_{\rm ph} \frac{L}{\lambda}
\end{equation}
where $L$ is the distance between source and observation point,
$\lambda$ is the wavelength and $\delta_{\rm ph}$ is the phase
slowness error associated with the particular angle of
propagation. Because it is proportional to $L / \lambda$
the phase error easily may become dominant if is not careful 
in the choice of discretization in the high-frequency regime.

\subsection{Dimensional reduction and Helmholtz-like symbols}

In case of coefficients, the symbol for arbitrary $h$ can be expressed
in terms of that for $h=1$
\begin{equation}
  P(\xi) = \frac{1}{h^2} P_1(h\xi ; hk)
\end{equation}
where 
\begin{equation}
  P_1(\xi,k) = \sum_\gamma e^{i\xi \cdot \gamma} f_\gamma(k) .
\end{equation}
By dimensional reduction the symbol $\tilde{Q}(\xi)$, $\hat{Q}(\xi)$ can be written as
\begin{equation}
  \tilde{Q}(\xi) = \tilde{Q}_1 (h\xi ; hk(x)) .
\end{equation}
where
\begin{equation}
  \tilde{Q}_1(\xi,k) = \sum_\gamma e^{i\xi \cdot \gamma} \tilde{g}_\gamma(k) .
\end{equation}
and similar for $\hat{Q}$.

To obtain the results below, we assume that the symbol $P$ is
{\em Helmholtz-like} as defined in the following
\begin{definition}
A symbol $P(\xi)$ is said to be Helmholtz-like if the zero set $Z_P$ can
be parameterized as in (\ref{eq:Z_star_shaped}), $Z_P$ is contained in
$]-\pi/h,\pi/h[^{d}$, $P(0)$ is negative, $\partial P/\partial \xi \neq 0$ at all
points in $Z_P$, and the map 
\begin{equation} \label{eq:define_map_N}
  N : Z_P \rightarrow S^{d-1} : 
  \xi \mapsto \frac{ \partial P/\partial \xi(\xi) }
    { \| \partial P / \partial\xi(\xi) \|}
\end{equation}
that maps a point in $Z_P$ to the unit normal to the surface is a
diffeomorphism. 
\end{definition}
It follows that $P$ is Helmholtz-like if $P_1$ is Helmholtz like.

\subsection{The discrete Green's function\label{subsec:discrete_Greensfunction_asymptotics}}

A Green's function $u(x)$ for the discrete
equation will be defined as a solution of
(\ref{eq:Greens_functions_general_spatial}). If $k$ is constant the
equivalent equation for the Fourier transform $U(\xi)$ reads
\begin{equation} \label{eq:Greens_function_general_Fourier}
  P(\xi) U (\xi) = 1 .
\end{equation}
We will first describe the general solution to this equation in the
Fourier domain. Then we will consider the asymptotics in the spatial
domain. Using the results obtained, we can then derive a unique outgoing
Green's function in the Fourier domain, and state its asymptotics.

Due to the zeros of $P$, problem (\ref{eq:Greens_function_general_Fourier})
has non-unique, distributional solutions. To explain their nature, we
recall the closely related one dimensional problem to determine all
$f$ such that
\begin{equation} \label{eq:simplified_distribution_problem}
  x f(x) = 1 .
\end{equation}
We also write this as $M_x f = 1$, where $M_x$ is the multiplication
operator by the function $x$. 
The solutions to (\ref{eq:simplified_distribution_problem})
are the distributions of the form \cite{Grubb2009}
\begin{equation} \label{eq:sol_simplified_distribution_problem}
  f(x) = \PV \frac{1}{x} + b \delta ,
\end{equation}
where $b$ is a free constant.
Here the distribution $\PV \frac{1}{x}$ is defined by
\begin{equation}
  \langle \PV \frac{1}{x} , \phi \rangle
  = 
  \lim_{\epsilon \rightarrow 0} 
    \int_{\RR \backslash [-\epsilon,\epsilon]} \frac{\phi(x)}{x} \, dx .
\end{equation}
In other words $\delta$ is in the kernel of $M_x$, while 
$\PV \frac{1}{x}$ is a particular solution to 
(\ref{eq:simplified_distribution_problem}).

In case of (\ref{eq:Greens_function_general_Fourier}) we similarly have a
nonzero kernel of $M_P$, with elements $B S_{Z_P}$, where $S_{Z_P}$ 
denotes the singular function of $Z_P$, which is the
distribution given by
\begin{equation}
  S_{Z_P}(\phi) = \int_{Z_P} \phi(x) \, dS(x) ,
\end{equation}
and $B$ is any distribution on $Z_P$.  For functions $f$ on $\RR^d$ 
with zero set $\tilde{Z}$ such that $\nabla f(y) \neq 0$ for all $y \in \tilde{Z}$, 
the principal value $u = \PV \frac{1}{f(y)}$ 
can be defined as follows. Let $\tilde{Z}_\epsilon = \{ x \,| \, 
d(x,\tilde{Z}) < \epsilon \}$, then
\begin{equation}
  \langle \PV \frac{1}{f(y)},\phi \rangle
  = \lim_{\epsilon \rightarrow 0}
  \int_{\RR^d \bs \tilde{Z}_\epsilon} \frac{\phi(y)}{f(y)} \, dy .
\end{equation}
We obtain

\begin{proposition} \label{prop:Fourier_general_sol}
The solutions to (\ref{eq:Greens_function_general_Fourier}) 
are given by
\begin{equation} \label{eq:general_inverse_Fourier_domain}
  U(\xi) = \PV \frac{1}{P(\xi)} + B S_{Z_P} 
\end{equation}
where $B$ can be any distribution on $Z_P$. 
\end{proposition}

The freedom in the choice of $B$ is related to the fact that in the
spatial domain one can add any linear combination of plane waves 
$e^{i x \cdot \xi}$ with $\xi \in Z_P$ and still have a solution.

Let $u(x)$ be the inverse Fourier transform of a solution $U(\xi)$
for some smooth $B$
\begin{equation} \label{eq:IFT_general_Greens_function}
  u(x) = (2\pi)^{-d} \int_{[-\pi/h,\pi/h]^d} 
    \PV \frac{1}{P(\xi)} e^{i x \cdot \xi} \, d\xi
    + (2\pi)^{-d} \int_{Z_P} B e^{i x \cdot \xi} .
\end{equation}
We will study the asymptotic behavior of this integral for large
$\|x\|$ using the method of stationary phase.  For
$p \in Z_P$, we define a certain curvature-like quantity $K(p)$ as
follows. After rotating the coordinates, we may assume that 
$\partial P/\partial \xi(p)$ is parallel to the $d$-th coordinate 
axis and that $Z_P$ is locally a graph
\begin{equation} \label{eq:parameterize_Z_xi_d}
  \xi_d = g(\xi_1,\ldots,\xi_{d-1}) .
\end{equation}
By the assumptions $g$ has a nondegenerate local
maximum at $(p_1,\ldots,p_{d-1})$.  
Let $-\lambda_1, \ldots, -\lambda_{d-1}$ denote the
eigenvalues of the second derivative matrix $\frac{\partial^2 g}
{\partial (\xi_1,\ldots,\xi_{d-1})^2}(p_1,\ldots,p_{d-1})$.  We define
\begin{equation} \label{eq:define_K_p}
  K(p) = \lambda_1 \lambda_2 \ldots \lambda_{d-1} .
\end{equation}
For $d=3$ this is the Gaussian curvature of the surface.
In the following proposition $N^{-1}$
denotes the inverse function of the map $N$ defined in
(\ref{eq:define_map_N}).
The result and its proof have some similarity with
results of Lighthill \cite{Lighthill1960}.

\begin{proposition} \label{prop:asymptotics}
Let $u$ be the inverse Fourier transform of a distribution $U$ as in
(\ref{eq:general_inverse_Fourier_domain}) such that
$B$ is a $C^\infty$ function on
$Z_P$. Let $\xi_+ = \xi_+(x) = N^{-1}(x / \|x\|)$ and $\xi_- =
\xi_-(x) = N^{-1}(-x/\|x\|)$. The function $u$ satisfies
\begin{equation} \label{eq:Greens_function_general_asymptotic}
\begin{aligned}
  u(x) = {}& (2\pi)^{-\frac{d+1}{2}} e^{-\frac{(d-1) \pi i}{4}}
   \|x\|^{-\frac{d-1}{2}} K(\xi_+)^{-1/2}  
    \bigg( \frac{\pi \, i}{\| \partial P/\partial\xi(\xi_+) \|} + B(\xi_+) \bigg)
 e^{i x \cdot \xi_+}
\\
  & + (2\pi)^{-\frac{d+1}{2}} e^{\frac{(d-1) \pi i}{4}}
   \|x\|^{-\frac{d-1}{2}} K(\xi_-)^{-1/2} 
    \bigg( - \frac{\pi \, i}{\| \partial P/\partial\xi(\xi_-) \|} + B(\xi_-) \bigg)
e^{i x \cdot \xi_-}
\\
&   + O(\|x\|^{-1/2-d/2}) , \qquad \|x\| \rightarrow  \infty .
\end{aligned}
\end{equation}
\end{proposition}

\begin{proof}
We start with the first integral in
(\ref{eq:IFT_general_Greens_function}). For $x \in (h \ZZ)^d$ the
domain is really a torus and the integrand is $C^\infty$ as a function
on the torus. It is convenient to replace the integral on the torus 
by an integral over a bounded subset of $\RR^d$. Let $\psi_1$ be a 
smooth, positive function 
supported in $[-\pi/h - \eta, \pi /h + \eta]$, that is one on
$]-\pi/h+\eta,\pi/h-\eta[$ and satisfies 
$\sum_{l=-\infty}^\infty \psi_1(x+ 2\pi l /h) = 1$ for $x \in \RR$, and let
\begin{equation}
  \psi(x) = \psi_1(x_1) \ldots \psi_1(x_d) 
\qquad x \in \RR^d
\end{equation}
Then we can write
\begin{equation} \label{eq:IFT_over_RRd}
  (2\pi)^{-d} \int_{[-\pi/h,\pi/h]^d} 
    \PV \frac{1}{P(\xi)} e^{i x \cdot \xi} \, d\xi
  = (2\pi)^{-d} \int_{\RR^d} \psi \PV \frac{1}{P_{\rm per}(\xi)} e^{i x \cdot \xi} \, d\xi
\end{equation}
for $x \in (h\ZZ)^d$, where
$P_{\rm per}$ is the periodic extension of $P$,
and this formula may also be considered for $x
\in \RR^d$. We assume $\eta$ is sufficiently small such that $Z_P$ is
supported in $ ]-\pi/h+\eta,\pi/h-\eta[ ^d$.

We will write $x = \tau v$, $v \in S^{d-1}$ and consider the limit
$\tau \rightarrow \infty$. We assume coordinates are rotated such that
$v = (0, \ldots, 0, 1)$, using the same notation for the new
coordinates as used so far for the old coordinates.

The integral on the right hand side of (\ref{eq:IFT_over_RRd}) will be
written as a sum of integrals over subsets using a partition of unity.
For some smooth cutoff function $\chi$, denote
\begin{equation}
  I_\chi 
  = (2\pi)^{-d} \int 
    \chi(\xi) \psi(\xi) \PV \frac{1}{P_{\rm per}(\xi)} 
    e^{i \tau \xi_d} \, d\xi
\end{equation}
We may assume there are four different types of $\chi$
\begin{enumerate}[(i)]
\item
$\chi = \chi_+$ is one on a neighborhood of $\xi_+$
\item
$\chi = \chi_-$ is one on a neighborhood of $\xi_-$
\item
on $\supp \chi \cap Z_P$ we can write $Z_P$ as a graph
$\xi_k = g(\xi_1,\ldots,\xi_{k-1},\xi_{k+1},\xi_d)$
\item
$\supp \chi \cap Z_P = \emptyset$
\end{enumerate}

We consider these four cases in the limit $\tau \rightarrow \infty$
using the method of stationary phase \cite{Duistermaat1996}.
In case (iv) the integral $I_\chi = O(\tau^{-N})$ for any $N$ by the
lemma of non-stationary phase and we
don't need to consider this case further.
In case (iii) we can write
\begin{equation}
  \chi \psi \PV \frac{1}{P_{\rm per}(\xi)} 
  = C(\xi) \PV \frac{1}{\xi_k -
    g(\xi_1,\ldots,\xi_{k-1},\xi_{k+1},\ldots,\xi_d)}
\end{equation}
for some smooth function $C(\xi)$
and perform the integral over $\xi_k$. This yields a smooth function
of $(\xi_1,\ldots,\xi_{k-1},\xi_{k+1},\xi_d)$. By the lemma
of non-stationary phase it follows that again $I_\chi = O(\tau^{-N})$.

In case (i) we can write $Z_P$ locally as a graph
$\xi_d = g(\xi_1, \ldots,\xi_{d-1})$. For brevity denote
$\xi' = (\xi_1, \ldots,\xi_{d-1})$
We observe that we can write
\begin{equation}
  \chi \psi \PV \frac{1}{P} =
  h_0(\xi)
  + h_1(\xi') \PV \frac{1}{\xi_d - g(\xi')}
  - h_1(\xi') 
    (1-\psi_2(\xi_d - g(\xi'))) \frac{1}{\xi_d - g(\xi')} .
\end{equation}
where $h_0, h_1, \psi$ are smooth, compactly supported functions and 
$\psi_2 = 1$ around
$0$ and $h_1(\xi_+') = \frac{1}{\| \partial P/\partial\xi(\xi_+) \|}$.
Then for the first term the lemma of non-stationary phase can be
invoked. Hence this term is $O(\tau^{-N})$ for any $N$.
For the third term, the same result can be obtained using integration
by parts. For the second part we recall the
standard Fourier transform
$\FourierF \PV \frac{1}{y} = -i \pi \sgn(\eta)$,
it follows that
\begin{equation}
  \FourierF^{-1} \PV \frac{1}{\eta - a} 
  = e^{i a y} \frac{i}{2} \sgn(y) .
\end{equation}
As a consequence, we obtain
\begin{equation} \label{eq:I_chi_plus_xiprime}
  I_{\chi_+} = (2\pi)^{-(d-1)} \frac{i}{2} 
  \int h_1(\xi') e^{i \tau g(\xi')} \, d\xi'
  + O(\tau^{-N})
\end{equation}
any $N$.
We can now apply the stationary phase lemma. The function $g$ has its
maximum at $\xi_+'$ and can be expanded as, possibly after a further
rotation of coordinates
\begin{equation}
  g(\xi_1,\ldots,\xi_{d-1})
  = v \cdot \xi_+ 
  - \sum_{j=1}^{d-1} \lambda_j (\xi_j - (\xi_{+})_j)^2 +
  O(\|\xi'-\xi_+'\|^3) ,
\end{equation}
see the discussion preceding (\ref{eq:define_K_p}). This yields
\begin{equation}
  I_{\chi_+} = \frac{i}{2} (2\pi)^{-(d-1)/2} \frac{1}{\| \partial P/\partial\xi(\xi_+) \|}
    e^{-(d-1) \frac{\pi i}{4}} e^{i\tau v \cdot \xi_+} \tau^{-(d-1)/2} K(\xi_+)^{-1/2} .
\end{equation}
The contribution $I_{\chi_-}$ in case (ii) can be computed similarly,
resulting in 
\begin{equation}
  I_{\chi_-} = - \frac{i}{2} (2\pi)^{-(d-1)/2} \frac{1}{\| \partial P/\partial\xi(\xi_-) \|}
    e^{(d-1) \frac{\pi i}{4}} e^{i\tau v \cdot \xi_-} \tau^{-(d-1)/2} K(\xi_-)^{-1/2} .
\end{equation}

For the surface integral
\begin{equation} \label{eq:surface_integral_prop2}
  (2\pi)^{-d} \int_{Z_P} B(\xi) e^{i x \cdot \xi} \, dS(\xi) ,
\end{equation}
we again assume $x = \tau v$ and consider the limit $\tau  \rightarrow
\infty$. A partition of unity is applied and, by the method of
stationary phase, the only contributions that are not $O(\tau^{-N})$
for any $N$ come from neighborhoods of $\xi_\pm$. To determine the
contribution from a neighborhood of $\xi_+$, 
we assume that $v$ is parallel to the $d$-th coordinate
axis so that $Z_P$ is locally given by a graph
$\xi_d = g(\xi')$, $\xi' = (\xi_1,\ldots,\xi_{d-1})$.
The method of stationary phase can be applied directly. The only
contributions come from neighborhoods of $\xi_\pm$, and can be
computed similarly as for the integral (\ref{eq:I_chi_plus_xiprime}).

The contribution $I_{\chi_+}$, $I_{\chi_-}$ and the two contributions
from the integral (\ref{eq:surface_integral_prop2}) together give the
result. 
\end{proof}

It is straightforward to obtain the outgoing solutions to
(\ref{eq:Greens_functions_general_spatial}) and
(\ref{eq:Greens_function_general_Fourier}). 
In (\ref{eq:Greens_function_general_asymptotic})
the term with phase factor $e^{i\tau v \cdot \xi_-}$ must vanish, and
we obtain the equation
\begin{equation}
  B(\xi) = \frac{\pi i}{\| \partial P/\partial\xi(\xi) \|} ,
    \qquad
    \text{for $\xi \in Z_P$.}
\end{equation}
We state this as a theorem and include the asymptotic expression for
the solution in the result.

\begin{theorem} \label{th:outgoing_asymptotics}
The outgoing solution to (\ref{eq:Greens_function_general_Fourier}) is given by
\begin{equation}
  U(\xi) = \PV \frac{1}{P(\xi)} + \frac{\pi i S_{Z_P}(\xi)}{\| \partial P/\partial\xi(\xi) \|} .
\end{equation}
Its inverse Fourier transform $u(x)$ satisfies
\begin{equation} \label{eq:asym_outgoing_solution}
\begin{aligned}
  u(x) = {}& (2\pi)^{-\frac{d-1}{2}} e^{-\frac{(d-1) \pi i}{4}}
   \|x\|^{-\frac{d-1}{2}}  
     \frac{i \, K(\xi_+)^{-1/2}}{\| \partial P/\partial\xi(\xi_+) \|} 
 e^{i x \cdot \xi_+}
  + O(\|x\|^{-1/2-d/2}) ,
\end{aligned}
\end{equation}
where $K$ and $\xi_\pm$ are as in proposition~\ref{prop:asymptotics}.
\end{theorem}

The above analysis can be repeated for continuous, Helmholtz like
operators with the same result (see also \cite{Lighthill1960}).
For the usual Helmholtz operator $H$ in $d=3$ dimensions we have that 
$Z_H$ given by $\| \xi \| = k$, $\| \partial H / \partial \xi \| 
\, |_{\xi\in Z_H} = 2 k$, 
$K_+ = \frac{1}{k^2}$ and 
\begin{equation}
  u(x) = \frac{1}{4\pi \|x\|} e^{i k \| x\|} + O(\|x\|^{-2}) ,
\end{equation}
so that the highest order asymptotic expansion actually equals the
well known outgoing Green's function.

\subsection{The modified discrete Green's function\label{subsec:modified_Greens_function}}
Let
\begin{equation}
  H_1(\xi,k) = \xi^2 - k^2 .
\end{equation}
It follows from theorem~\ref{th:outgoing_asymptotics} that if $H_1$
and $P_1$ have the same zero sets, i.e.\ identical phase slownesses,
then the solutions to $H u = \delta$ and $P u = \delta$ have
asymptotically the same phase. The amplitudes however will differ by a
factor $\frac{\| \partial H_1/\partial\xi \|}{\| \partial P_1/\partial\xi
  \|} $ evaluated at the zero set. 
In this subsection we consider therefore the solutions $u$ to the equations
(\ref{eq:Qdelta_source}) and (\ref{eq:Qoperator_to_Qdelta_GF}), which,
as we will see, obtain different amplitudes.

The Fourier transformed solution $U(\xi)$ to (\ref{eq:Qdelta_source}) 
and (\ref{eq:Qoperator_to_Qdelta_GF}) is given by the product of the solution $U$
given in proposition~\ref{prop:Fourier_general_sol} and a factor 
$\tilde{Q}(\xi) \hat{Q}(\xi)$. 
Using this, we can formulate a result similar to
Theorem~\ref{th:outgoing_asymptotics}. In this case the adjective
outgoing refers to the solution $v$ of (\ref{eq:Qdelta_source}). The
result can be proven by similar arguments as used to prove 
proposition~\ref{prop:asymptotics} and theorem~\ref{th:outgoing_asymptotics}.

\begin{theorem} \label{th:modified_Greens_function}
The Fourier transform of the outgoing solution to
(\ref{eq:Qdelta_source}) and (\ref{eq:Qoperator_to_Qdelta_GF}) is
given by
\begin{equation}
  U(\xi) = \tilde{Q}(\xi) \hat{Q}(\xi) \left( 
\PV \frac{1}{P(\xi)} + \frac{\pi i S_{Z_P}(\xi)}{\| \partial P/\partial\xi(\xi) \|} .
\right) .
\end{equation}
Its inverse Fourier transform $u(x)$ satisfies
\begin{equation} 
\begin{aligned}
  u(x) = {}& (2\pi)^{-\frac{d-1}{2}} e^{-\frac{(d-1) \pi i}{4}}
   \|x\|^{-\frac{d-1}{2}}  
     \frac{i \, \tilde{Q}(\xi_+) \hat{Q}(\xi_+) 
    K(\xi_+)^{-1/2}}{\| \partial P/\partial\xi(\xi_+) \|}  e^{i x \cdot \xi_+}
  + O(\|x\|^{-1/2-d/2}) ,
\end{aligned}
\end{equation}
where $K$ and $\xi_\pm$ are as in proposition~\ref{prop:asymptotics}.
\end{theorem}

Summarizing our findings so far, the discrete solutions $u$ to
(\ref{eq:Qdelta_source}) and (\ref{eq:Qoperator_to_Qdelta_GF}) are
asymptotically equal to the solutions of the continuous Helmholtz
equation $Hu = \delta$ if the following two conditions are satisfied
\begin{enumerate}[(i)] 
\item \label{it:same_zero_set}
$P_1(\xi,k)$ and $H_1(\xi,k)$ have the same zero sets
\item \label{it:condition_tildeQ_hatQ}
$\hat{Q}_1$ and $\tilde{Q}_1$ satisfy
\begin{equation} \label{eq:QLQR_nablaP_over_nablaH}
 \tilde{Q}_1(\xi,k) \hat{Q}_1(\xi,k) 
  = \frac{\| \partial P_1 / \partial \xi (\xi,k)\|}
        {\| \partial H_1 / \partial \xi (\xi,k)\|}
\end{equation}
for all $(\xi,k)$ such that $P_1(\xi,k) = H_1(\xi,k) = 0$
\end{enumerate}

\section{Theory of discrete Helmholtz equations with variable $k$%
\label{sec:theory_variable}}

In this section we define a class of discrete approximations to the
Helmholtz operator with variable $k$, together with the associated
symbols. This is the topic of
subsection~\ref{eq:symbols_ops_variable_k}. We then study
ray-theoretic solutions to the equation
(\ref{eq:Greens_functions_general_spatial}) and to the set of
equations (\ref{eq:Qdelta_source}), (\ref{eq:Qoperator_to_Qdelta_GF}),
where $P$ and $Q$ are now variable coefficient operators.

We assume that $k = \frac{\omega}{c(x)}$, where $c$ is smooth and we
consider the limit $\omega \rightarrow \infty$. In the discrete case
we assume that $h \omega = \text{constant}$.  Ray-theoretic solutions
are then based on the ansatz
\begin{equation} \label{eq:ray_theory_ansatz}
  u(x,\omega) = A(x,\omega) e^{i \omega \Phi(x)} , 
\end{equation}
for some smoothly varying $A$ and $\Phi$.
For the continuous Helmholtz equation, such solutions are well
known and are constructed in two steps. First the ansatz
(\ref{eq:ray_theory_ansatz}) is inserted in the PDE, and an
expansion in $\omega$ is performed. Requiring that the highest
order terms vanish leads to the eikonal equation for $\Phi$ and the
transport equation for $A$. Secondly, initial/boundary conditions for these
equations are obtained from the asymptotic behavior of the constant
coefficient solutions. In this way, the solution modulo an error of
lower order in $\omega$ is obtained.

For our class of difference equations we follow the same program.
The constant coefficient solutions were already analyzed in
subsection~\ref{subsec:discrete_Greensfunction_asymptotics}.
In subsection~\ref{subsec:discrete_raytheory} we find a
nonlinear first order PDE for $\Phi$ and a transport equation for $A$.
Remarkably, we obtain the same equations in the continuous and
discrete case when formulated in terms of
the symbols (which are defined for both continuous and discrete
problems). See \cite{Lighthill1960} and
\cite{Duistermaat1996} for the continuous case and methods used in
that case as well as here.

In the last part of this section we consider
the ray-theoretic solutions to (\ref{eq:Qdelta_source}) and
(\ref{eq:Qoperator_to_Qdelta_GF}). The conditions 
(\ref{it:same_zero_set}) and
(\ref{it:condition_tildeQ_hatQ}) from
subsection~\ref{subsec:modified_Greens_function} 
for $P$ and $Q$ to obtain accurate solutions, 
need to be modified and extended to have the same ray-theoretic phase
and amplitude in the continuous and discrete case. The operator $P$
should be discretized using a symmetric discretization (with $=1/2$,
see below) and we should have $\tilde{Q} = \hat{Q}$. This is the topic
of subsection~\ref{subsec:amplitude_correction}.

\subsection{Symbols and operators for variable $k$\label{eq:symbols_ops_variable_k}}

In case $k$ depends on $x$, finite difference discretizations of the
Helmholtz operator may depend in different ways on the function
$k$. For example, the coefficients $p_{\alpha,\beta}$ may depend on
$k$ and its derivatives at $x = h \alpha$, but they may also depend on
$k$ at different points, for example on $k(h\alpha)$ and
$k(h\beta)$. We will consider a class of difference operators $P$,
where the matrix elements $p_{\alpha,\beta}$ depend only on the value
of $k$ at $(1-t)h\alpha + t h \beta$, where $t \in \{ 0 ,1/2, 1\}$
is a fixed constant.  In other words we consider operators $P$ with
matrix elements of the form
\begin{equation} \label{eq:define_P_variable_k}
  p_{\alpha,\beta} = \frac{1}{h^2} 
    f_{\alpha - \beta}(h k((1-t)\alpha h + t \beta h)) .
\end{equation}
Note that the operator is symmetric if $t = 1/2$ and $f_\gamma =
f_{- \gamma}$. This will turn out to be an appropriate choice for a 
discrete Helmholtz operator. We will assume that $k(x)$
is defined for all $x$, not only those in the grid.
Similar we assume that for $\tilde{Q}$ we have
\begin{equation} 
  \tilde{q}_{\alpha,\beta} = 
    \tilde{g}_{\alpha - \beta}(h k((1-t)\alpha h + t \beta h)) .
\end{equation}
and similar for $\hat{Q}$.

For such operators it is not obvious how to define the symbol. To find
an appropriate definition, we first consider how to define an operator
from a symbol $H(x,\xi)$ in the continuous case. This is the subject
of pseudodifferential operator theory, and can be done with the formula
\cite{AlinhacGerard2007,Hormander1985}
\begin{equation} \label{eq:quantization_continuous}
  \operatorname{Op}_t(H(x,\xi)) u 
  = (2\pi)^{-d} \iint H(x + t(y-x),\xi) e^{i(x-y) \cdot \xi}  u(y) \,
  d\xi \, d y .
\end{equation}
A map from a function $H(x,\xi)$ to an operator such as
$\operatorname{Op}_t(H(x,\xi))$ is called a
quantization. For $t = 0$, the previous formula is the standard 
left-quantization, $t=1$ is the right quantization and $t=1/2$ 
is the Weyl quantization. If $H(x,\xi) = \xi^2 - k(x)^2$, then
$\operatorname{Op}_t(H(x,\xi)) = -\Delta -k(x)^2$, independently of
which of these quantizations is used.

To obtain a symbol associated with the operator $P$ defined in 
(\ref{eq:define_P_variable_k}) we rewrite the expression for $Pu(x)$ 
as follows
\begin{equation} \label{eq:rewrite_Pvar_step1}
\begin{aligned}
  P u(x) = {}&
    h^{-2} \sum_\gamma f_\gamma(hk(x+t \gamma h)) u(x+h\gamma) 
\\
  = {}& h^{-2+d} \sum_\gamma \sum_{y \in (h\ZZ)^d}
    f_\gamma(hk(x + t(y-x))) \delta(x+h\gamma - y) u(y)  .
\end{aligned}
\end{equation}
Using the Fourier domain representation
$\delta(x) = (2\pi)^{-d} \int_{[-\pi/h,\pi/h]^d} e^{i x \cdot \xi}$
this can be rewritten as
\begin{equation}
  Pu(x)  = h^{-2+d} (2\pi)^{-d} \sum_{y \in (h\ZZ)^d} \int_{[-\pi/h,\pi/h]^d} \sum_\gamma 
    f_\gamma(hk(x+t(y-x)))
    e^{i (x+h\gamma -y )\cdot \xi}  \, d\xi .
\end{equation}
This can be written in similar form as
(\ref{eq:quantization_continuous}), namely as
\begin{equation}
  \operatorname{Op}_t(P(x,\xi)) u (x)
  \stackrel{\rm def}{=}
  (2\pi)^{-d} \sum_{y \in (h\ZZ)^d} \int_{[-\pi/h,\pi/h]^d} 
   P(x+t(y-x),\xi) e^{i (x-y)\cdot \xi} u(y) \, d\xi ,
\end{equation}
where 
\begin{equation} \label{eq:symbol_discrete_variable_k}
  P(x,\xi) = h^{-2} \sum_\gamma f_\gamma(h k(x)) e^{i h\gamma \cdot \xi} .
\end{equation}
Thus, associated with $P$ defined in
(\ref{eq:define_P_variable_k}) is associated the symbol $P(x,\xi)$
given in (\ref{eq:symbol_discrete_variable_k}). The parameter $t$
corresponds to the type of quantization, left, right or Weyl quantization.

With these definitions, the symbol
(\ref{eq:symbol_discrete_variable_k}) 
for variable coefficients
may also be expressed entirely
in terms of $P_1$
\begin{equation}
  P(x,\xi) = \frac{1}{h^2} P_1 (h\xi ; hk(x)) .
\end{equation}
A symbol $P(x,\xi)$ is called Helmholtz like if it satisfies the
definition for each fixed $x$.

\subsection{Ray-theoretic equations for amplitude and phase%
\label{subsec:discrete_raytheory}}

In this section we consider the high-frequency limit $\omega
\rightarrow \infty$. We assume that $\omega h = \text{constant}$,
and recall that $k(x) = \frac{\omega}{c(x)}$, where $c(x)$ is $C^\infty$.
The operator $P$ and the symbol $P(x,\xi)$ become $\omega$-dependent.
By $\tilde{P}(x,\xi)$ we denote the symbol for $\omega = 1$.
\begin{equation}
  \tilde{P}(x,\xi) = 
  \frac{1}{(\omega h)^2} P_1( h\omega \xi, \frac{h
    \omega}{c(x)} ) .
\end{equation}
For other values of $\omega$ we find that
\begin{equation} \label{eq:scaling_P}
  P(x,\xi ; \omega)
  = \omega^2 \tilde{P}(x, \frac{\xi}{\omega})
\end{equation}

We consider the action of $P$ on functions of the form
\begin{equation}
  u(x) = e^{i \omega \Phi(x)} A(x)
\end{equation}
where $\Phi$ and $A$ are $C^\infty$ functions. 
From the symbol $P(x,\xi;\omega)$ and the phase function 
$\Phi$ one can derive naturally a
vector field, which we call $L_{P,\Phi,\omega}$
(cf.\ \cite[section 4.3]{Duistermaat1996})
\begin{equation}
  \big( L_{P,\Phi,\omega} \big)_j = \pdpd{P}{\xi_j}(x,\omega\nabla \Phi;\omega) .
\end{equation}
This vector field is determined by 
$L_{\tilde{P},\Phi} = \pdpd{\tilde{P}}{\xi}(x,\nabla \Phi)$ as follows
\begin{equation} \label{eq:scaling_L}
  L_{P,\Phi,\omega}(x) 
    = \omega L_{\tilde{P},\Phi}(x) .
\end{equation}

\begin{proposition} \label{prop:discrete_Duis4_3_2}
We have 
\begin{equation} \label{eq:WKB_like_discrete}
\begin{aligned}
  e^{ - i \omega \Phi(x)}
    P (e^{i\omega \Phi(x)} A( x))
  = & \omega^2 \tilde{P}(x,\nabla \Phi(x)) A(x) 
\\
  & + \omega \frac{1}{i} \bigg(
    \sum_j (L_{\tilde{P},\Phi})_j  \pdpd{A}{x_j} 
    + \frac{1}{2} ( \operatorname{div} L_{\tilde{P},\Phi}) A
    + (t-1/2) \sum_j \pdpd{^2 \tilde{P}}{x_j \partial \xi_j} A \bigg)
\\
  & + O(1)  , \qquad \omega \rightarrow \infty .
\end{aligned}
\end{equation}
\end{proposition}

\begin{proof}
The proof uses a Taylor expansion of the phase function to second order
\begin{equation}
  \Phi(x+y)
    = \Phi(x)
      + \nabla \Phi(x) \cdot y
      + \frac{1}{2} \sum_{j,k} \pdpd{^2 \Phi}{x_j x_k} y_j y_k 
      + O(\|y\|^3) ,
\end{equation}
a Taylor expansion of the amplitude to first order 
\begin{equation}
  A(x+y) = A(x)
    + \nabla A(x) \cdot y + O(\|y\|^2))
\end{equation}
and a Taylor expansion of the matrix coefficients to first order
\begin{equation}
  p_{\alpha,\beta} 
  = \frac{1}{h^2} f_{\alpha-\beta}( hk(h\alpha) ) 
    + \frac{t}{h^2} f_{\alpha-\beta}' ( hk(h\alpha) ) 
      h \nabla k (h \alpha) \cdot h(\beta-\alpha) .
\end{equation}
The exponent $e^{i \omega \Phi(x+y)}$ is then written
as a product of three factors
\begin{equation}
  e^{i \omega \Phi(x+y)}
  = 
  e^{i \omega \Phi(x)}
  e^{i \omega \nabla \Phi(x) \cdot y}
\left( 1 + \frac{1}{2} i \omega \sum_{j,k} \pdpd{^2 \Phi}{x_j x_k}
  y_j y_k + O(\omega \|y\|^3) \right) .
\end{equation}
These expansions are inserted in the sum
\begin{equation}
  (P u)(x) 
  = \sum_\gamma \frac{1}{h^2} f_\gamma(h k( x + th\gamma) )
  u(x +h\gamma ) . 
\end{equation}
The factor $ e^{i \omega \Phi(x)}$ can be
put in front of the expression outside the summation
\begin{equation}
\begin{aligned}
  (P u)( x ) = & 
  e^{ i \omega \Phi(x)} \frac{1}{h^2} \bigg[ 
    \sum_{\gamma}
    e^{i \omega \nabla \Phi(x) \cdot h\gamma}
    f_\gamma( hk(x)) A(x) 
\\
  & + \bigg(
    \sum_{\gamma} e^{i \omega \nabla \Phi(x) \cdot h\gamma}
    f_\gamma( hk(x)) \nabla A(x)  \cdot (h\gamma)
\\
  & \qquad + \frac{1}{2}  i \omega  A
    \sum_\gamma e^{i \omega \nabla \Phi(x) \cdot h\gamma}  
    \sum_{j,k} \pdpd{^2 \Phi}{x_j \partial x_k} (h\gamma)_j
        (h\gamma)_k f_\gamma( hk(x)) 
\\ & \qquad
    + A \sum_\gamma t e^{i \omega \nabla \Phi(x) \cdot h\gamma}
    f_\gamma'  h\nabla k  \cdot (h \gamma) \bigg)
\\
  & + O(h^2)  \bigg] 
\end{aligned}
\end{equation}
We next use the expression for the symbol
$P(x,\xi) = \frac{1}{h^2} \sum_\gamma e^{i h \gamma \cdot \xi} f_\gamma(hk(x))$,
the following expressions for the derivatives of $P(x,\xi)$
\begin{align}
  \pdpd{P}{\xi_j} = {}& \frac{i}{h^2} 
    \sum_\gamma h \gamma_j e^{i h\gamma \cdot \xi} f_\gamma(hk(x))
\\
  \pdpd{^2 P}{x_j \partial \xi_k} = {}&
    \frac{i}{h^2} \sum_\gamma (x_\gamma)_k e^{i x_\gamma \cdot \xi}
    f_\gamma'(k(x)) h \pdpd{k}{x_j}
\\
  \pdpd{^2 P}{\xi_j \partial \xi_k}  = {}&
  - \frac{1}{h^2} \sum_\gamma ( h\gamma)_j (h \gamma)_k e^{i h \gamma \cdot \xi}
    f_\gamma(hk(x)) .
\end{align}
and an expression for the derivatives of $L_{P,\Phi,\omega}$
\begin{equation}
  \pdpd{(L_{P,\Phi,\omega})_j}{x_k}
  = \pdpd{^2 P}{x_k \partial \xi_j}(x,\omega \nabla \Phi)
    + \omega \sum_l \pdpd{^2 P}{\xi_j \partial \xi_l} \pdpd{^2
      \Phi}{x_l \partial x_k} .
\end{equation}
This yields
\begin{equation}
\begin{aligned}
  (P u)( x ) = & 
  e^{ i \omega \Phi(x)} \bigg[ P(x,\omega\nabla\Phi(x);\omega) A(x) 
\\
  & + \frac{1}{i} \bigg( \sum_{j=1}^n (L_{P,\Phi,\omega})_j
  \pdpd{A}{x_j}
    + \frac{1}{2} \operatorname{div} L_{P,\Phi,\omega} A
    + (t-\frac{1}{2}) \sum_j \pdpd{^2 P}{x_j \xi_j} A  \bigg)
\\
  & + O(1)  \bigg] , \qquad \omega \rightarrow \infty .
\end{aligned}
\end{equation}
Using equations (\ref{eq:scaling_P}) and (\ref{eq:scaling_L}) the
result follows.
\end{proof}

The result is similar to the result in Proposition 4.3.2
of \cite{Duistermaat1996}. To find $A$ and $\Phi$ such 
that
\begin{equation}
  P(e^{i \omega \Phi(x)} A(x)) \approx 0
\end{equation}
the phase function $\Phi$ must satisfy the equation
\begin{equation}
  \tilde{P}(x,\nabla \Phi) = 0 ,
\end{equation}
which is a nonlinear first order equation like an eikonal equation,
and the amplitude must satisfy a transport equation
\begin{equation}
  \sum_j (L_{\tilde{P},\Phi})_j  \pdpd{A}{x_j} 
    + \frac{1}{2} ( \operatorname{div} L_{\tilde{P},\Phi}) A
    + (t-1/2) \sum_j \pdpd{^2 \tilde{P}}{x_j \partial \xi_j} A  = 0
\end{equation}
For $t = 1/2$, this equation conserves $|A|^2$.

Ray theoretic solutions to equation
(\ref{eq:Greens_functions_general_spatial}) can now be constructed
just as in the continuous case. By a rescaling,
theorem~\ref{th:outgoing_asymptotics} can be used to obtain the
asymptotics of a solution for $x \neq 0$ and $\omega \rightarrow
\infty$. The amplitude and phase from formula
(\ref{eq:asym_outgoing_solution}) can hence be used as
initial/boundary values for the eikonal equation for $\Phi$ and the
transport equation for $A$, and these $\Phi$ and $A$ can be determined
from these equations, where we note that the eikonal equation may not
have globally defined solutions, just as in the continuous case.

We briefly recall the continuous equivalent of
Proposition~\ref{prop:discrete_Duis4_3_2}.  The following is basically
a reformulation of proposition 4.3.2 of \cite{Duistermaat1996} and can
be proven using the method of stationary phase found in the same text.

\begin{proposition}
Let $H$ be a continuous Helmholtz like symbol.
For the action of $\operatorname{Op}(H)$ on 
$e^{i \omega \Phi} A(x)$ we have the
asymptotic development
\begin{equation} \label{eq:WKB_like_continuous}
\begin{aligned}
  e^{-i \omega \Phi(x)}
  \operatorname{Op}_t (H) (e^{i \omega \Phi(x)} A(x))
  = & \omega^2 \tilde{H}(x,\nabla \Phi(x)) A(x) 
\\
  & + \omega \frac{1}{i} \bigg(
    \sum_j (L_{\tilde{H},\Phi})_j  \pdpd{A}{x_j} 
    + \frac{1}{2} ( \operatorname{div} L_{\tilde{H},\Phi}) A
    + (t-1/2) \sum_j \pdpd{^2 \tilde{H}}{x_j \partial \xi_j} A \bigg)
\\
  & + O(1)  \bigg] , \qquad \omega \rightarrow \infty .
\end{aligned}
\end{equation}
\end{proposition}

\subsection{Amplitude correction\label{subsec:amplitude_correction}}

In this section we consider ray-theoretic approximations $v = e^{i
  \omega \Psi} B$ and $u = e^{i \omega \Phi} A$ for the solutions $v$
and $u$ to (\ref{eq:Qdelta_source}) and (\ref{eq:Qoperator_to_Qdelta_GF}).
Assume we have reference ray-theoretic solutions
$u_{\rm ref} = e^{i \omega \Phi_{\rm ref}(x)} A_{\rm ref}(x)$ 
associated with Helmholtz equation $H u_{\rm ref} = \delta$ where
$H(\omega) = -\Delta - \frac{\omega^2}{c(x)^2}$, i.e.\ 
$\Phi_{\rm ref}$ satisfies the eikonal equation and $A_{\rm ref}$
the transport equation with appropriate initial conditions.
In the following we let $H_1(\xi,k) = \xi^2 - k^2$, 
$\tilde{H}(x,\xi) = \xi^2 - \frac{1}{c(x)^2}$.

Assume that
\begin{enumerate}[(i)]
\item
$P_1(\xi,k)$ and $H_1(\xi,k)$ have the same zero sets
\item
$\tilde{Q}_1$ and $\hat{Q}_1$ are identical and 
$Q_1(\xi,k) = \tilde{Q}_1(\xi,k) = \hat{Q}_1(\xi,k)$ satisfies
\begin{equation}
  Q_1(\xi,k)^2 = 
    \frac{\| \partial P_1 / \partial \xi (\xi,k)\|}
        {\| \partial H_1 / \partial \xi (\xi,k)\|}
\end{equation}
for all $(\xi,k)$ such that $P_1(\xi,k) = H_1(\xi,k) = 0$;
\item 
$P$ and $H$ are derived from their respective symbols using $t
= 1/2$ quantization .
\end{enumerate}
{\em We argue that in this case, to highest order $u$ has the same
ray-theoretic approximation as $u_{\rm ref}$. } We omit a formal proof,
because the arguments are similar as those used above.

The construction of the phase and amplitude functions $\Psi$ and $B$
proceeds almost in the same way as for solutions to
(\ref{eq:Greens_functions_general_spatial}). Eikonal and transport
equations are as follows from 
Proposition~\ref{prop:discrete_Duis4_3_2}. The constant coefficient
solutions differ by a factor $Q(\xi_+)$ and have the same phase,
resulting in different initial/boundary conditions, such that on a
small sphere $\Gamma$ around 0, where we impose the initial/boundary
conditions for the eikonal and transport equations, we have
\begin{equation} \label{eq:initial_conditions_Psi_B}
\begin{aligned} 
  \Psi(x) = {}& \Phi_{\rm ref}(x)
\\
  B(x) = {}& Q_{\omega=1}(x,\nabla \Phi_{\rm ref}) \,  A_{\rm ref}(x) 
\end{aligned}
\end{equation}
As a result, $\Psi(x) = \Phi_{\rm ref}(x)$ everywhere.
While we have different transport equations the operators $L_{\tilde{H},\Phi_{\rm ref}}$
and $L_{\tilde{P},\Phi_{\rm ref}}$ are scaled versions of each other
\begin{equation}
  L_{\tilde{P},\Phi_{\rm ref}} = Q_{\omega=1} (x,\nabla \Phi_{\rm ref})^2 
    L_{\tilde{H},\Phi_{\rm ref},1}
\end{equation}
It follows from this fact and the transport equation for $A_{\rm
  ref}$, that 
\begin{equation}
  \sum_j L_{\tilde{P},\Phi_{\rm ref}} \pdpd{Q_{\omega=1}^{-1} A_{\rm ref}}{x_j} 
  + (\operatorname{div} L_{\tilde{P},\Phi_{\rm ref}}) Q_{\omega=1}^{-1} A_{\rm ref} = 0 .
\end{equation}
This and (\ref{eq:initial_conditions_Psi_B}) shows that
\begin{equation}
  B(x) = Q^{-1}(x,\nabla \Phi_{\rm ref}) \,  A_{\rm ref}(x) 
\end{equation}
everywhere.

The function $u$ is given by applying $Q$ to $v$.
The action of $Q$ on $e^{i\omega \Psi} B$ is to highest order equal to
a multiplication by $Q(x,\omega \nabla \Psi ; \omega)$, so that 
\begin{equation}
\begin{aligned}
  \Phi = {}& \Psi = \Phi_{\rm ref} \qquad \text{and}
\\
  A(x) = {}& \tilde{Q}(x,\nabla \Phi_{\rm ref}) B(x) = A_{\rm ref}(x) ,
\end{aligned}
\end{equation}
concluding the argument.

\section{Phase slowness errors for existing discretizations}
\label{sec:phase_slowness_errors_existing}

In this section we will describe three types of discretizations of the
Helmholtz equation (\ref{eq:Helmholtz_equation}), namely standard
finite differences, compact finite differences and Lagrange
finite elements on regular meshes. We then compute phase slowness
errors to compare the performance of the different methods in this
respect, and to obtain reference values for our new method constructed
below.

We modify the notation compared to the previous two section.  In this
section the degrees of freedom for all three types of methods are
denoted by $u_{j,k,l}$ (in three dimensions) and associated with a
regular mesh with grid spacing $h$. For finite element methods of
order $N$ the cells are of size $N h$, and cell boundaries are located
at $j,k,l \equiv 0 \mod N$. Occasionally we will use $d = 2$ or $3$ to
denote the dimension of space.

\subsection{Standard finite differences}
In a standard finite difference discretization of the operator
$-\Delta - k^2$ each of the one-dimensional second derivatives in the
Laplacian $\Delta = \pdpd{^2}{x_1^2} + \pdpd{^2}{x_2^2} +
\pdpd{^2}{x_3^2}$ is approximated by a central difference
approximation of the given order. These are given by
\begin{equation}
  D_2^{(N)} u_l = h^{-2} \sum_{m=-N/2}^{N/2} c^{(N)}_m u_{l+m}
\end{equation}
where the $c_m^{(N)}$ are as in the following table for $N = 2,4,6,8$
\[
\begin{tabular}{r|ccccccccc} 
& $m\!=\!-4$ & -3 & -2 & -1 & 0 & 1 & 2 & 3 & 4 \\ \hline
$N=2$ &
& & & 1 & -2 & 1 & & & 
\\[0.7ex]
4 & 
& & $-\frac{1}{12}$ & $\frac{4}{3}$ & $-\frac{5}{2}$  & $\frac{4}{3}$ & $-\frac{1}{12}$ & & 
\\[0.7ex]
6 &
& $\frac{1}{90}$ & $-\frac{3}{20}$ & $\frac{3}{2}$ &
$-\frac{49}{18}$  & $\frac{3}{2}$ & $-\frac{3}{20}$ & $\frac{1}{90}$ &
\\[0.7ex]
8 & 
$-\frac{1}{560}$ & $\frac{8}{315}$ & $-\frac{1}{5}$  & $\frac{8}{5}$ &
$-\frac{205}{72}$ & $\frac{8}{5}$ & $-\frac{1}{5}$  & $\frac{8}{315}$ &
$-\frac{1}{560}$ 
\end{tabular}
\]
The discrete approximation to the term $-k(x)^2 u$ in
(\ref{eq:Helmholtz_equation}) is simply given by $-k_{l,m,n}^2
u_{l,m,n}$. The two-dimensional case can be done similarly.

\subsection{Compact finite difference discretizations\label{subsec:compact_fd}}
For constant $k$, compact finite difference discretizations
take the form
\begin{equation} \label{eq:matrix_compact_3D}
  (A u)_{l,m,n} = \sum_{(p,q,r) \in \{-1,0,1\}^3} a_{p,q,r} u_{l+p,m+q,n+r} .
\end{equation}
Because of symmetry, there are four different coefficients $A_j$,
$j=0,1,2,3$ and
\begin{equation} \label{eq:coefficients_A_j}
  a_{p,q,r} = A_{|p| + |q| + |r|} , \qquad (p,q,r) \in \{-1,0,1\}^3
\end{equation}
In 2-D we have
\begin{equation}
  (A u)_{l,m} = \sum_{(p,q) \in \{-1,0,1\}^2} a_{p,q} u_{l+p,m+q} .
\end{equation}
and there are three different coefficients $A_j$, $j=0,1,2$ and
\begin{equation}
  a_{p,q} = A_{|p| + |q|}  , \qquad (p,q) \in \{-1,0,1\}^2.
\end{equation}

The choice of coefficients is done in different ways in 
\cite{BabuskaEtAl1995,HarariTurkel1995,JoShinSuh1996,OpertoEtAl2007,
  Sutmann2007,ChenChengWu2012,TurkelEtAl2013,StolkEtAl2014}.
The QS-FEM method \cite{BabuskaEtAl1995} is a two-dimensional
method, for which the coefficients are given modulo an overall
normalization by
\begin{equation}
\begin{aligned}
  A_0^{\rm QS-FEM} = {}& 4
\\
  A_1^{\rm QS-FEM} = {}& 2 \frac{c_1(\alpha) s_1(\alpha) - c_2(\alpha) s_2(\alpha)}
                {c_2(\alpha) s_2(\alpha) (c_1(\alpha)+s_1(\alpha))
                - c_1(\alpha) s_1(\alpha) (c_2(\alpha)+s_2(\alpha))}
\\
  A_2^{\rm QS-FEM} = {}& \frac{c_2(\alpha)+s_2(\alpha) - c_1(\alpha)-s_1(\alpha)}
                {c_2(\alpha) s_2(\alpha) (c_1(\alpha)+s_1(\alpha))
                - c_1(\alpha) s_1(\alpha) (c_2(\alpha)+s_2(\alpha))} 
\end{aligned}   
\end{equation}       
where $\alpha = kh$ and the auxiliary functions $c_1,s_1,c_2,s_2$ are
defined by
\begin{equation}
\begin{aligned}
  c_1(\alpha) = {}& \cos \left(\alpha \cos \frac{\pi}{16} \right)
& 
  s_1(\alpha) = {}& \cos \left(\alpha \sin \frac{\pi}{16} \right)
\\
  c_2(\alpha) = {}& \cos \left(\alpha \cos \frac{3 \pi}{16} \right)
& 
  s_2(\alpha) = {}& \cos \left(\alpha \sin \frac{3 \pi}{16} \right) .
\end{aligned}
\end{equation}
We will not discuss the fourth order method of \cite{HarariTurkel1995}
because it contains still a free parameter and one of the authors has
later published a sixth order method in \cite{TurkelEtAl2013}. In the
latter method variations of $k$ are taken into account. In case of
constant $k$, the coefficients are given in three dimensions by
\begin{equation}
\begin{aligned}
  A_0^{\rm CHO6} = {}& +\frac{64}{15} - \frac{14 k^2 h^2}{15}
                + \frac{k^4 h^4}{20} 
\\
  A_1^{\rm CHO6} = {}& -\frac{7}{15} + \frac{k^2 h^2}{90}
\qquad
  A_2^{\rm CHO6} = -\frac{1}{10} - \frac{k^2 h^2}{90}
\qquad
  A_3^{\rm CHO6} = -\frac{1}{30}
\end{aligned}
\end{equation}
and in two dimensions by
\begin{equation}
\begin{aligned}
  A_0^{\rm CHO6} = {}& \frac{10}{3} - \frac{41 k^2 h^2}{45} + \frac{k^4 h^4}{20}
\\
  A_1^{\rm CHO6} = {}& -\frac{1}{6} - \frac{k^2 h^2}{90} 
\qquad
  A_2^{\rm CHO6} = - \frac{2}{3} - \frac{k^2 h^2}{90} ,
\end{aligned}
\end{equation}
again modulo an overall constant. For the method of Sutmann
\cite{Sutmann2007} we have (this method is only for 3-D)
\begin{equation}
\begin{aligned}
  A_0^{\rm SUT} = {}& \frac{64}{15} 
    \left( 1 - \tfrac{1}{4} k^2 h^2 + \tfrac{5}{256} k^4 h^4 
        - \tfrac{1}{1536} k^6 h^6 \right) 
\\
  A_1^{\rm SUT} = {}& -\frac{7}{15} 
    \left( 1 - \tfrac{1}{21} k^2 h^2 \right) 
\qquad
  A_2^{\rm SUT} = -\frac{1}{10} 
    \left( 1 +  \tfrac{1}{18} k^2 h^2 \right) 
\qquad
  A_3^{\rm SUT} = -\frac{1}{30}
\end{aligned}
\end{equation}

In \cite{JoShinSuh1996,OpertoEtAl2007,ChenChengWu2012,StolkEtAl2014}
the contributions to $-\Delta - k^2$ are split in a contribution from
$-\Delta$ and a contribution from $-k^2$. We define a three
dimensional, symmetric discretization of the identity $M$ depending on
three parameters $\alpha_1,\alpha_2,\alpha_3$ by
\begin{equation} \label{eq:Helmholtz_family_M}
\begin{aligned}
  (M u)_{l,m,n} = {}& \sum_{(p,q,r) \in \{-1,0,1\}^3} m_{p,q,r} u_{l+p,m+q,n+r}
\\
  m_{p,q,r} = {}& M_{|p| + |q| + |r|} , \qquad (p,q,r) \in \{-1,0,1\}^3
\end{aligned}
\end{equation}
where now $M_0 = \alpha_1$, $M_1 = \frac{\alpha_2}{6}$,
$M_2 = \frac{\alpha_3}{12}$, $M_3 =
\frac{1-\alpha_1-\alpha_2-\alpha_3}{8}$. For constant $k$ the
discrete form of the term $-k^2 u$ is given by $-k^2 (M u)_{l,m,n}$.
Before defining the negative Laplacian we define a two-dimensional
weighting operator, discretizing the identity, 
depending on two additional parameters $\alpha_4,\alpha_5$
\begin{equation} \label{eq:Helmholtz_family_N}
\begin{aligned}
  (N^{[1,2]} u)_{l,m,n} = {}& \sum_{{p,q} \in \{-1,0,1\}^2} n_{p,q} u_{l+p,m+q,n}
\\
  n_{p,q} = {}& N_{|p| + |q|} , \qquad (p,q) \in \{-1,0,1\}^2
\end{aligned}
\end{equation}
with $N_0 = \alpha_4$, $N_1 = \frac{\alpha_5}{4}$,
$N_2 = \frac{1-\alpha_4-\alpha_5}{4}$.
Each of the second deratives in the Helmholtz operator will be
discretized using the tensor product of a two-dimensional weighting
operator that discretizes the identity
and the standard second order discrete second derivative 
$D_2^{[j]}$, where $j=1,2$ or $3$ indicates along which axis the
second derivative operators. The resulting matrix is
\begin{equation} \label{eq:Helmholtz_family}
   - D_2^{[1]} \otimes N^{[2,3]}
   - D_2^{[2]} \otimes N^{[1,3]}
   - D_2^{[3]} \otimes N^{[1,2]}
   - k^2 M 
\end{equation}
In the two-dimensional case there are in total three parameters
$\alpha_1,\alpha_2,\alpha_3$, with
$M_0 = \alpha_1$, $M_1 = \frac{\alpha_2}{4}$,
$M_2 = \frac{1-\alpha_1-\alpha_2}{4}$, and $N_0 = \alpha_3$,
$N_1 = \frac{1-\alpha_3}{2}$.
The coefficients $A_j$ for the 3-D case are given in terms of the 
$\alpha_j$ (modulo an overall constant $h^{-2})$ by
\begin{equation} \label{eq:from_alpha_to_A_3D}
\begin{aligned}
  A_0 = {}&   6 \alpha_4 - (k h)^2 \alpha_1 
&
  A_2 = {}& - \tfrac{1}{2} \alpha_5 + \tfrac{1}{2} (1-\alpha_4-\alpha_5)
     - (k h)^2 \tfrac{1}{12} \alpha_3
\\
  A_1 = {}& - \alpha_4 + \alpha_5 - (k h)^2 \tfrac{1}{6} \alpha_2
&
  A_3 = {}& - \tfrac{3}{4} (1-\alpha_4-\alpha_5) 
    - (k h)^2 \tfrac{1}{8} (1 - \alpha_1-\alpha_2-\alpha_3) .
\end{aligned}
\end{equation}
For the 2-D case we have
\begin{equation} \label{eq:from_alpha_to_A_2D}
\begin{aligned}
  A_0 = {}&   4 \alpha_3 - (k h)^2 \alpha_1 
&
  A_1 = {}&   1 - 2\alpha_3 - (k h)^2 \tfrac{1}{4} \alpha_2
&
  A_2 = {}&  -1 + \alpha_3 - (k h)^2 \tfrac{1}{4} (1-\alpha_1-\alpha_2) .
\end{aligned}
\end{equation}

An advantage of this formulation using tensor products is that in case
of PML layers aligned with the coordinate axes, the second order
operator $D_2^{[j]}$ can simply be replaced by its PML-modified
version%
\footnote{%
In a PML layer, say a layer associated with $x_1 = \text{constant}$,
the derivative $\pdpd{}{x_1}$ is replaced by a 
$\alpha_{\rm PML,1}(x_1) \pdpd{}{x_1}$ where
$\alpha_{\rm PML,1}(x_1) = \frac{1}{1 + i \sigma_1(x_1) / \omega}$ and
the function $\sigma_1$ indicates the local amount of damping 
\cite{Johnson_notespml,Berenger1994,ChewWeedon1994}. We
choose $\sigma_1$ quadratically increasing. The discrete second
derivative in the first coordinate in this PML layer becomes
$h^{-2} \alpha_{\rm PML,1}(x_{l,m,n})
  ( \alpha_{\rm PML,1}(x_{i+1/2,j,k}) (u_{l+1,m,n}-u_{l,m,n})
  - \alpha_{\rm PML,1}(x_{i-1/2,j,k}) (u_{l,m,n}-u_{l-1,m,n}) )$
By rescaling the equations with a factor $\alpha_{\rm PML,1}^{-1}$
the symmetry of the system is restored.}.

In this way we have derived a family of second order accurate
discretizations. In \cite{JoShinSuh1996} and \cite{OpertoEtAl2007}
the same family of discretizations is considered in two resp.\ three
dimensions (but differently parameterized), and coefficients are
chosen such that the maximum phase slowness error is minimized, where
the maximum is taken over all angles and a
range of $kh$ corresponding to at least four points per wavelength. 
This leads to the choices
\begin{equation}
\begin{aligned}
  \alpha_1^{\rm OPT} = {}& 0.4964958 
&
  \alpha_2^{\rm OPT} = {}& 0.4510125
&
  \alpha_3^{\rm OPT} = {}& 0.052487
\\
  \alpha_4^{\rm OPT} = {}& 0.648355362
&
  \alpha_5^{\rm OPT} = {}& 0.296692332
\end{aligned}
\end{equation}
for the method of \cite{OpertoEtAl2007} and 
\begin{equation}
\begin{aligned}
  \alpha_1^{\rm JSS} = {}& 0.6248 
&
  \alpha_2^{\rm JSS} = {}& 0.37524
& 
  \alpha_3^{\rm JSS} = {}& 0.77305
\end{aligned}
\end{equation}
for the method of \cite{JoShinSuh1996}.

In \cite{ChenChengWu2012} and \cite{StolkEtAl2014} it is observed that
smaller phase slowness errors are obtained when the parameters $\alpha_j$
are allowed to vary.  In \cite{ChenChengWu2012} a set of 7 parameters
(in three dimensions) is chosen piecewise constant. We will not
describe this method in detail but refer to the paper for resulting
phase errors.  In \cite{StolkEtAl2014} the above described set of 5
parameters are chosen as piecewise linear functions. However, in this
work, the aim is different, because the phase slowness
differences with a fine scale operator are minimized, not with the
exact operator, so that the values of the phase slowness errors cannot
be compared.

\subsection{Lagrange finite elements on regular meshes}
For the description of Lagrange finite elements on regular meshes,
which will also be used in some of the numerical examples, we start
with the one-dimensional case. In this case the finite element cells
are the intervals $((j-1) Nh, jNh)$, $j=1,2,\ldots$, each containing
$N-1$ interior points and two boundary points. The reference cell is
$(0,N)$, and shape functions on this reference cell are given by
standard Lagrange polynomials, which we will denote by $L^{(N)}_j(x)$,
and which are one at $x=j$ and zero at $x = 0 , 1 ,\ldots, j-1,j+1,
\ldots N-1,N$ and defined to be 0 outside $[0,N]$.  Letting $k \in \ZZ
$ and $l \in \{ 0, 1, \ldots, N-1\}$, the one-dimensional trial and
test functions are given by
\begin{equation}
  \psi_{kN+l}^{(N)}(x) = \left\{ \begin{array}{ll}
  L_0^{(N)}(\frac{x}{h}-(k-1)N) + L_N^{(N)}(\frac{x}{h}-kN) & \text{if } l=0\\ 
  L_l^{(N)}(\frac{x}{h}-kNh) & \text{otherwise}
\end{array}
\right .
\end{equation}
The two and three-dimensional trial and testfunctions are given by
tensor products of the $\psi^{(N)}_j$.
The finite element discretizion is of course derived from the weak form
\begin{equation} \label{eq:weak_form}
  \Phi(u,v)  \stackrel{\rm def}{=}
  \int_\Omega \sum_{j=1}^d \pdpd{u}{x_j}\pdpd{v}{x_j} \, dx
    - \int_\Omega k^2 u v \, dx = \int_\Omega f v \, dx .
\end{equation}
The elements of the matrix in the discretization are given by
\begin{equation}
  a_{l,m,n; p,q,r}
  = 
  \Phi(\psi_{p,q,r}, \psi_{l,m,n} )
\end{equation}
The contribution from the term $\int_\Omega \sum_{j=1}^d
\pdpd{u}{x_j}\pdpd{v}{x_j} \, dx$ can be called the stiffness matrix and
the contribution from $\int_\Omega k^2 u v \, dx$ can be called the mass
matrix. 
If $k$ is constant (or cellwise constant) , the stiffness and
mass matrices can be computed exactly. If $k$ is variable, then only the
stiffness matrix can be computed exactly, and for the mass matrix some
sort of quadrature must be used. For constant $k$ these computations
are standard and easily done using a computer algebra
system, and we will not write down the resulting coefficients.

For constant $k$, the finite difference methods are obviously
translationally symmetric, i.e.\ if we denote by $a_{l,m,n; p,q,r}$ the matrix elements we have
\begin{equation} \label{eq:translation_symmetry}
  a_{l,m,n; p,q,r} = a_{l+A,m+B,n+C; p+A,q+B,r+C}
\end{equation}
For the finite elements there is a symmetry under a subset of
translations given by the $A,B,C$ that are multiples of $N$.

\subsection{Phase slowness errors}

For finite difference methods, 
finding the phase velocities or slownesses comes down to
determining the zeros of the symbol $P(\xi)$ associated with 
a difference operator $P$, see 
subsection~\ref{subsec:symbol_phase_slowness}.
The symbol is not difficult to obtain, for example, for 
compact finite difference discretizations, the symbol is
\begin{equation}
\begin{aligned}
  {}& \qquad P(\xi) = h^{-2} \big[ A_0 
  + 2A_1 (\cos(h\xi_1) + \cos(h\xi_2) + \cos(h\xi_3)) 
  + 2A_2 (\cos(h(\xi_1+\xi_2)) + \cos(h(\xi_1-\xi_2))
\\
  {}&
  +\cos(h(\xi_1+\xi_3)) + \cos(h(\xi_1-\xi_3)) 
  + \cos(h(\xi_2+\xi_3)) + \cos(h(\xi_2-\xi_3)))
\\
  {}& + 2A_3 (  \cos(h(\xi_1+\xi_2+\xi_3)) + \cos(h(\xi_1-\xi_2+\xi_3)) 
  + \cos(h(\xi_1+\xi_2-\xi_3)) + \cos(h(\xi_1-\xi_2-\xi_3))) \big]
\end{aligned}
\end{equation}
To compute the zeros numerically, the standard numerical solver fsolve
from Matlab was used, as well as the more accurate version
vpasolve. 

For finite element methods the elements of the kernel of the operator
are no longer simple plane waves, but Bloch waves. In the appendix is
described how we compute the phase slowness errors in this case.

Phase slowness errors are directionally dependent, i.e. they depend on
$\theta \in S^{d-1}$ as explained in
section~\ref{subsec:symbol_phase_slowness}.
They also depend on $kh$ or equivalently on the number of points per
wavelength $G = \frac{2\pi}{kh}$.
We have computed the maximum relative
phase slowness errors over $\theta \in S^{d-1}$ for
a number of schemes as a function of $1/G$.
In two and three dimensions these schemes are the
finite element schemes of order $1,2,3,4,6$ and $8$, the standard
finite difference discretizations of order $2,4,6$ and $8$ and the
sixth order compact method of \cite{TurkelEtAl2013}. In the
graphs below these results will be indicated by the letter FE1, FE2
etc., FD2, FD4, etc.\, and CHO6. In two dimensions we also included
results for the
QS-FEM method of \cite{BabuskaEtAl1995} and the method of Jo, Shin and
Suh \cite{JoShinSuh1996}, denoted by JSS. In three dimensions we also
have the method of Operto Et Al \cite{OpertoEtAl2007}, indicated by
OPT4 and the method of Sutmann \cite{Sutmann2007}, indicated by SUT.  
We have not included results on the method of
\cite{ChenChengWu2012}, these are given in Figure 2(c) in that
work. The phase slowness errors as a function of
$1/G$ are plotted in Figure~\ref{fig:phase_errors_exist_2D}
and~\ref{fig:phase_errors_exist_3D}.
At the end of section~\ref{sec:IOFD} we will briefly discuss these results.

\begin{figure}
\begin{center}
\includegraphics[width=10cm]{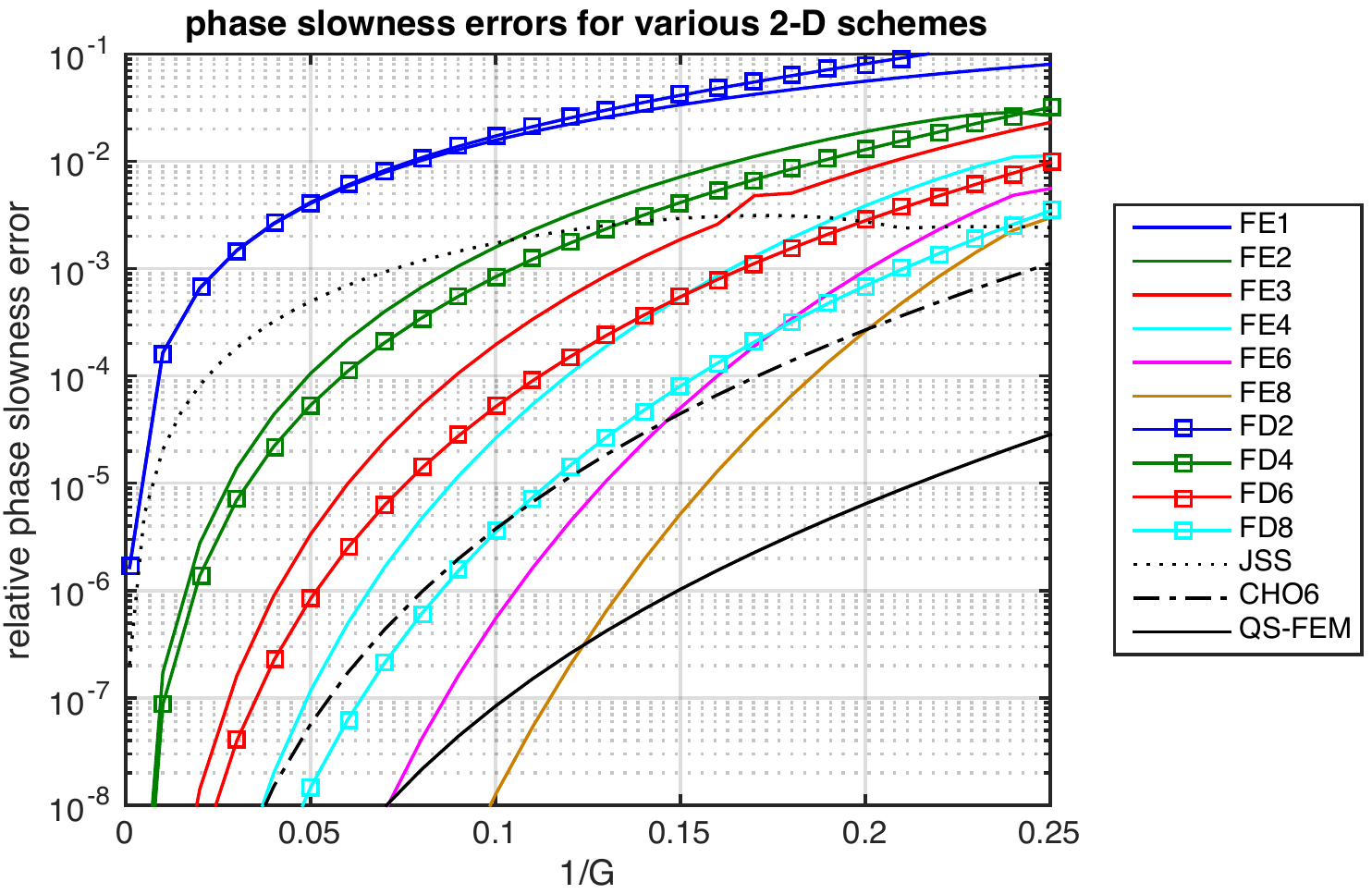}
\end{center}
\caption{Phase slowness errors for some 2-D schemes 
  as a function of the inverse number of points per wavelength $1/G$.}
\label{fig:phase_errors_exist_2D}
\end{figure}
\begin{figure}
\begin{center}
\includegraphics[width=10cm]{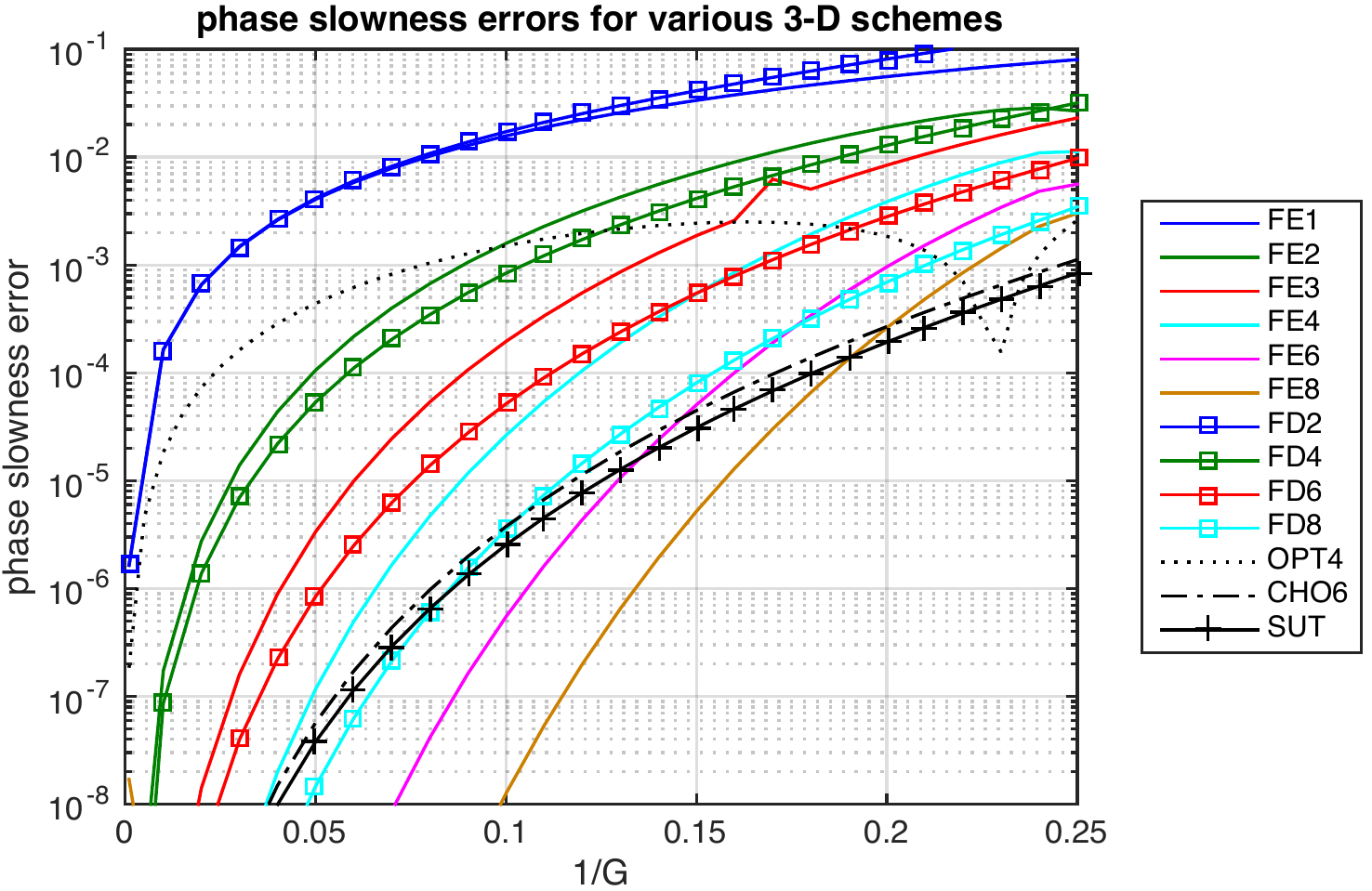}
\end{center}
\caption{Phase slowness errors for some 3-D schemes
  as a function of the inverse number of points per wavelength $1/G$.}
\label{fig:phase_errors_exist_3D}
\end{figure}

\section{A dispersion minimizing scheme with amplitude corrections}
\label{sec:IOFD}

In this section we will define our new discretization of the Helmholtz
equation. In this scheme, the approximate solution to the Helmholtz
equation $H u = f$, $H = -\Delta -k(x)^2$ is found by solving a
discrete system
\begin{equation} \label{eq:discrete_system_abstract1}
  P v = Q f
\end{equation}
and then setting
\begin{equation} \label{eq:discrete_system_abstract2}
  u = Q v ,
\end{equation}
where $P$ and $Q$ are compact finite difference operators defined
momentarily.

In section~\ref{sec:theory_variable} we studied ray-theoretic
solutions to difference equations of the type
(\ref{eq:discrete_system_abstract1}) and
(\ref{eq:discrete_system_abstract2}) with $f = \delta$, 
and we observed that the ray-theoretic solution to these equation
would be identical to those of the Helmholtz equation
\begin{equation}
  H u = \delta, \qquad 
  H = -\Delta - k(x)^2
\end{equation}
if the requirements (i) to (iii) of 
subsection~\ref{subsec:amplitude_correction}
are satisfied.
It follows from the derivations that if these properties are not satisfied exactly,
but there are small differences between the zero set of $P_1$ and that
of $H_1$ and between the values of $Q_1(\xi,k)^2$ and
$\frac{\partial P_1/\partial\xi(\xi,k)}{\partial H_1/\partial \xi(\xi,k)}$
then there will be small
errors in the phase and amplitude of the ray-theoretic
solutions. The operators $P$ and $Q$ will be chosen such that these differences
are minimal.
We will first construct $P$ in
subsection~\ref{subsec:IOFD_Helmholtz}. Then $Q$ will be constructed
in subsection~\ref{subsec:construct_Q}. In subsection we will discuss
the phase errors of the new method.

\subsection{IOFD discretization of the Helmholtz operator\label{subsec:IOFD_Helmholtz}}
In sections~\ref{sec:theory_constant} and~\ref{sec:theory_variable} a
general form for $P$ was given in terms of functions $f_\gamma$ of $kh$, see
(\ref{eq:define_P_constant}), (\ref{eq:define_P_variable_k}). 
For the $5$ or $3$ dimensional operator family  of 
subsection~\ref{subsec:compact_fd} (in 3 and 2 dimensions
respectively), these are given by
\begin{equation}
  f_{\gamma} = \left\{ \begin{array}{ll}
        6 \alpha_4 - (k h)^2 \alpha_1 
            & \text{for $\gamma \in \{ -1,0,1\}^3$, $|\gamma| = 0$}\\
        - \alpha_4 + \alpha_5 - (k h)^2 \tfrac{1}{6} \alpha_2
            & \text{for $\gamma \in \{ -1,0,1\}^3$, $|\gamma| = 1$}\\
- \tfrac{1}{2} \alpha_5 + \tfrac{1}{2} (1-\alpha_4-\alpha_5)
     - (k h)^2 \tfrac{1}{12} \alpha_3
            & \text{for $\gamma \in \{ -1,0,1\}^3$, $|\gamma| = 2$}\\
- \tfrac{3}{4} (1-\alpha_4-\alpha_5) 
    - (k h)^2 \tfrac{1}{8} (1 - \alpha_1-\alpha_2-\alpha_3)
            & \text{for $\gamma \in \{ -1,0,1\}^3$, $|\gamma| = 3$,}
  \end{array} \right. 
\end{equation}
in 3-D and by
\begin{equation}
  f_{\gamma} = \left\{ \begin{array}{ll}
      4 \alpha_3 - (k h)^2 \alpha_1 
            & \text{for $\gamma \in \{ -1,0,1\}^2$, $|\gamma| = 0$}\\
      1 - 2\alpha_3 - (k h)^2 \tfrac{1}{4} \alpha_2
            & \text{for $\gamma \in \{ -1,0,1\}^2$, $|\gamma| = 1$}\\
      -1 + \alpha_3 - (k h)^2 \tfrac{1}{4} (1-\alpha_1-\alpha_2)
            & \text{for $\gamma \in \{ -1,0,1\}^2$, $|\gamma| = 2$.}
  \end{array} \right. 
\end{equation}
in 2-D. Here $|\gamma| = |\gamma_1| + \ldots + |\gamma_d|$ and we used
equations (\ref {eq:from_alpha_to_A_3D}) and 
(\ref{eq:from_alpha_to_A_2D}). We let $\alpha_j$ depend on 
$\frac{hk}{2\pi} = 1/G$, where $G$ is the number of points per
wavelength used in the discretization
\begin{equation}
  \alpha_j = \alpha_j(1/G) , \qquad 1/G = \frac{k h}{2\pi} .
\end{equation}
Next we will choose a parameterization for these function and we will
describe how, by minimizing the phase slowness errors in a least-squares
sense, we obtain suitable choices of the functions $\alpha_j$,
$j = 1,\ldots, 2d-1$.

In \cite{StolkEtAl2014} the $\alpha_j$ where chosen to depend
piecewise linearly on $1/G$.  Here we let $\alpha_j$
depend piecewise polynomially on $1/G$, using Hermite interpolation.
We will specify a number of control nodes, and at each node the value of
$\alpha_j$ and its first derivative $\pdpd{\alpha_j}{(1/G)}$ are prescribed.
We will assume that the coefficients $\alpha_j$ vary slowly, so that 
we can indeed define the four coefficient of the stencil using five 
parameters depending on $1/G$. If $n_{\rm C}$ denotes the number of
control nodes, in this way the functions $\alpha_j$ are parameterized
by $2 n_{\rm C}$ parameters. In total we have $(4d-2) n_{\rm C}$
parameters, collectively denoted by $P$.

Next we specify the objective functional. The first contribution to
the objective functional is the square integrated phase slowness
error, integrated over angle and $1/G$.
Because of the symmetries, the phase slowness error need not be
integrated over all $\theta \in S^{d-1}$, but can be integrated over a
subset $\Theta_d$ of the sphere. In 2 dimensions, the angle variable
can be chosen in $\Theta_2 = [0,\pi/4]$. In 3 dimensions, using
spherical coordinates $\theta = (\theta_1,\theta_2) 
= \text{( polar angle, azimuthal angle)}$, the domain is
$\Theta_3=  [0,\pi/2] \times [0,\pi/4]$.
The second contribution to the objective functional 
is a regularization term involving $\frac{d \alpha_j}{d(1/G)}$.
In summary, we have
\begin{equation} \label{eq:objective_functional}
  T(P) = \int_0^{1/G_{\rm max}} \int_{\Theta_d} | \delta_{\rm ph}(\theta,P) |^2 \, d \theta \, d(1/G) 
  + \lambda \int _0^{1/G_{\rm max}} \sum_{j=1}^{2d-1} 
    \left| \frac{d \alpha_j}{d(1/G)} \right|^2 \, d(1/G) .
\end{equation}
This integral is discretized, using a weighted sum of regularly
sampled contributions.
By $n_{\rm A}$ we denote the number of angles to discretize the
integration over the sphere and by $n_{\rm G}$ the
number of choices for the parameter $1/G$. In two dimensions we chose
$n_{\rm A}=20$, in three dimensions $n_{\rm A} = 200$. The $1/G$ axis
was discretized in steps of $0.01$ in the integral
(\ref{eq:objective_functional}).
The value of $\lambda = 10^{-12}$ was used and
control nodes where chosen in the interval $[0,
0.4]$ with distance $0.05$. 

The objective functional was minimized using the Matlab function
lsqnonlin, aimed particularly at least-squares problems.  We found
that the optimization problem using the least squares objective
functional converges better than other types of objective functionals,
such as a sup-norm. This was done in three steps. First the control 
values for $1/G$ in $[0,0.2]$ were determined, then for $1/G$ in $[0.1,0.3]$
keep the values for $1/G < 0.1$ equal to those already obtained, and
then for $1/G$ in $[0.2,0.4]$ keeping those for $1/G < 0.2$ already
obtained. In this way somewhat better result were obtained than when
minimization was done directly for $G \in [0,0.4]$.
The parameters were determined heuristically, in such a way that
increasing the number of discretization points would not yield
substantial improvements. The phase slowness errors for $1/G \in
[0,0.1]$ were most sensitive to details of the method. As the phase
speed errors are very small in this parameter range we have not
explored this further.

The results of the optimization for the two- and three-dimensional
case are given in Tables~\ref{tab:coeff_IOFD_2D} and
\ref{tab:coeff_IOFD_3D}.  The phase slowness errors as a function of
$1/G$ (maximum over angle) are given in
Figure~\ref{fig:IOFD_phase_errors}, together with those of the IOFD
and CHO6 methods.

\begin{table}
\begin{center}
\small
\begin{tabular}{c|cccccc}
$1/G$ 
& $\alpha_1$ & $\pdpd{\alpha_1}{(1/G)}$
& $\alpha_2$ & $\pdpd{\alpha_2}{(1/G)}$
& $\alpha_3$ & $\pdpd{\alpha_3}{(1/G)}$ \\ \hline
0.00 &  0.702988 &  0.009776 &  0.260661 & -0.017374 &  0.833321 & -0.000611 \\ 
0.05 &  0.705833 & -0.009915 &  0.253348 & -0.046566 &  0.832408 & -0.036116 \\ 
0.10 &  0.704294 & -0.053006 &  0.251395 & -0.029803 &  0.829828 & -0.066179 \\ 
0.15 &  0.700617 & -0.097783 &  0.250099 & -0.016222 &  0.825956 & -0.087744 \\ 
0.20 &  0.694664 & -0.144215 &  0.249306 & -0.010052 &  0.821312 & -0.096545 \\ 
0.25 &  0.686959 & -0.169986 &  0.247309 & -0.061204 &  0.817120 & -0.066627 \\ 
0.30 &  0.677167 & -0.227359 &  0.243807 & -0.072388 &  0.815138 & -0.008931 \\ 
0.35 &  0.664000 & -0.306018 &  0.239969 & -0.074632 &  0.816970 &  0.085964 \\ 
0.40 &  0.645668 & -0.434744 &  0.237317 & -0.026502 &  0.823706 &  0.183724 \\ 
\hline
\end{tabular}%
\end{center}
\caption{Coefficients two-dimensional IOFD}
\label{tab:coeff_IOFD_2D}
\end{table}
\begin{table}
\begin{center}\small
\hspace*{-2cm}%
\begin{tabular}{c|cccccccccc}
$1/G$ 
& $\alpha_1$ & $\pdpd{\alpha_1}{(1/G)}$
& $\alpha_2$ & $\pdpd{\alpha_2}{(1/G)}$
& $\alpha_3$ & $\pdpd{\alpha_3}{(1/G)}$
& $\alpha_4$ & $\pdpd{\alpha_4}{(1/G)}$
& $\alpha_5$ & $\pdpd{\alpha_5}{(1/G)}$ \\ \hline
0.0000 &  0.635413 & -0.000228 &  0.210638 &  0.016303 &  0.172254 & -0.014072 
      &  0.710633 & -0.006278 &  0.245303 &  0.019576 \\ 
0.0500 &  0.635102 & -0.015578 &  0.210152 & -0.023424 &  0.171912 & -0.005802 
      &  0.709821 & -0.047764 &  0.245148 &  0.021398 \\ 
0.1000 &  0.634166 & -0.034804 &  0.208167 & -0.043396 &  0.171146 & -0.012462 
      &  0.707374 & -0.070981 &  0.244762 &  0.007493 \\ 
0.1500 &  0.632093 & -0.054496 &  0.205348 & -0.065935 &  0.170031 & -0.022145 
      &  0.703359 & -0.088202 &  0.245160 &  0.009937 \\ 
0.2000 &  0.628341 & -0.103457 &  0.201605 & -0.069385 &  0.169740 &  0.001893 
      &  0.698813 & -0.092327 &  0.245687 &  0.012201 \\ 
0.2500 &  0.622526 & -0.133896 &  0.197423 & -0.098212 &  0.169475 & -0.002559 
      &  0.694726 & -0.066617 &  0.246454 &  0.016791 \\ 
0.3000 &  0.614611 & -0.183988 &  0.192414 & -0.115398 &  0.168690 & -0.005589 
      &  0.692615 & -0.011177 &  0.247743 &  0.029213 \\ 
0.3500 &  0.603680 & -0.255991 &  0.186819 & -0.120930 &  0.167581 & -0.015564 
      &  0.694109 &  0.077605 &  0.250098 &  0.059733 \\ 
0.4000 &  0.588498 & -0.356326 &  0.180737 & -0.132266 &  0.166640 & -0.001852 
      &  0.700902 &  0.199685 &  0.254352 &  0.106049 \\ 
\hline
\end{tabular}%
\hspace*{-2cm}
\end{center}
\caption{Coefficients three-dimensional IOFD}
\label{tab:coeff_IOFD_3D}
\end{table}
\begin{figure}
\begin{center}
\includegraphics[width=10cm]{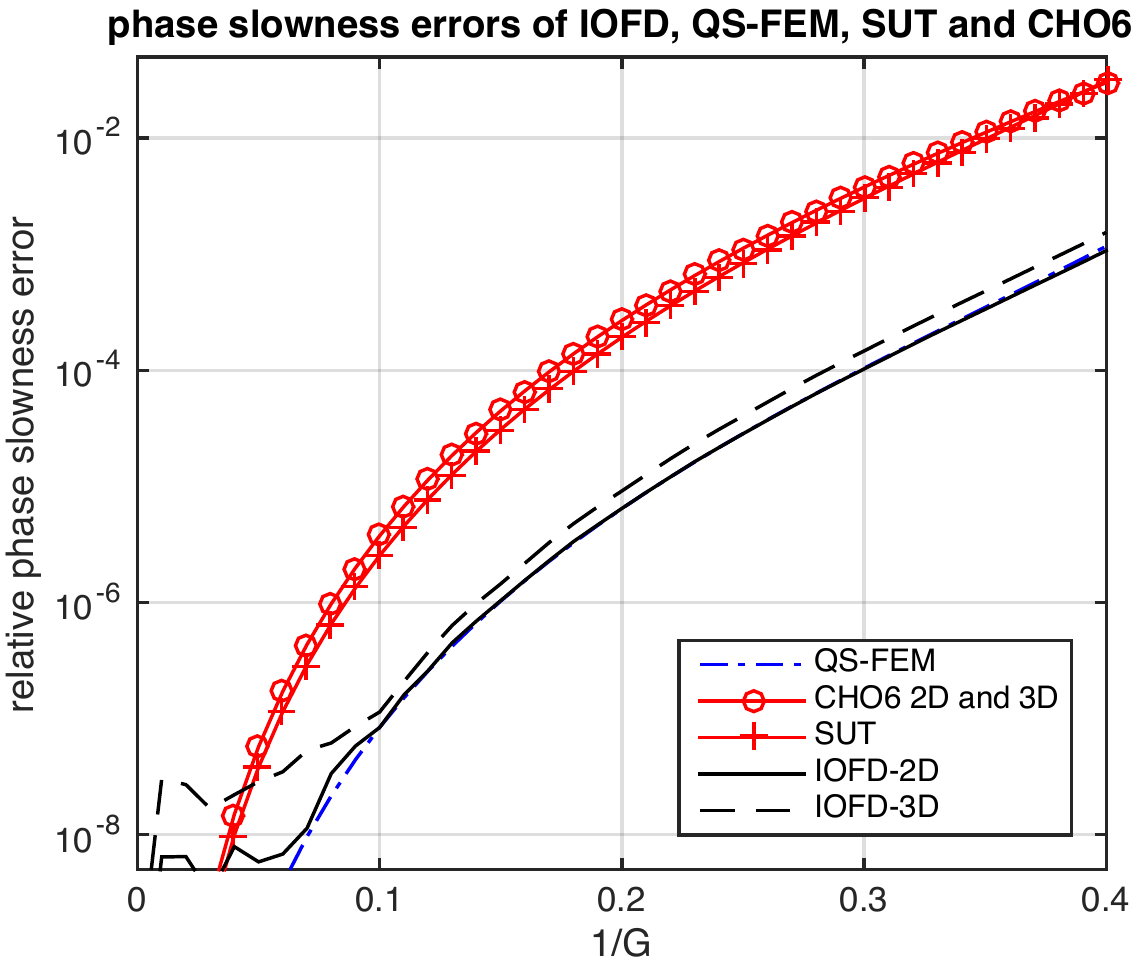}
\end{center}
\caption{Phase slowness errors for the IOFD method compared
  to QS-FEM and the sixth order method of \cite{Sutmann2007} and
  \cite{TurkelEtAl2013}.}
\label{fig:IOFD_phase_errors}
\end{figure}

\subsection{Amplitude correction operators\label{subsec:construct_Q}}

Here we will construct the difference operators $Q$. We will use the
second order discretizations of the identity defined in
(\ref{eq:Helmholtz_family_M}). Here we will denote the coefficients of
this family by $\beta_j$, which will be functions of
$\frac{kh}{2\pi} = 1/G$. Thus the functions $g_\gamma$ associated with
$Q$, see (\ref{eq:define_matrix_elements_q}) are given by
\begin{equation}
  g_{\gamma} = \left\{ \begin{array}{ll}
    \beta_1
        & \text{for $\gamma \in \{ -1,0,1\}^3$, $|\gamma| = 0$}\\
    \frac{\beta_2}{6}
        & \text{for $\gamma \in \{ -1,0,1\}^3$, $|\gamma| = 1$}\\
    \frac{\beta_3}{12}
        & \text{for $\gamma \in \{ -1,0,1\}^3$, $|\gamma| = 2$}\\
    \frac{1-\beta_1-\beta_2-\beta_3}{8}
        & \text{for $\gamma \in \{ -1,0,1\}^3$, $|\gamma| = 3$,}
      \end{array} \right.
\end{equation}
in 3-D and by
\begin{equation}
  g_{\gamma} = \left\{ \begin{array}{ll}
    \beta_1
            & \text{for $\gamma \in \{ -1,0,1\}^2$, $|\gamma| = 0$}\\
    \frac{\beta_2}{4}
            & \text{for $\gamma \in \{ -1,0,1\}^2$, $|\gamma| = 1$}\\
    \frac{1-\beta_1-\beta_2}{4}
            & \text{for $\gamma \in \{ -1,0,1\}^2$, $|\gamma| = 2$,}
    \end{array} \right.
\end{equation}
in 2-D, where $\beta_j = \beta_j(1/G)$. The $\beta_j(1/G)$ will be defined by Hermite interpolation from control
values similarly as we did for the $\alpha_j(1/G)$.
The control values are chosen to minimize 
a discrete approximation of the integral
\begin{equation}
  \int_0^{1/G_{\rm max}}
    \int_{\Theta_d} \left[ Q(\xi) 
        - \sqrt{ \frac{\| \partial P/\partial\xi(\xi)\|}
    {\| \partial H/\partial \xi(\xi)\| } }
        \right]^2_{\xi = \omega s_{\rm ph}(\theta) \theta} 
    \, d\theta
   \, d (1/G) .
\end{equation}
This integral is discretized in the same way as in the previous
subsection. This results in a linearly constrained linear least 
squares problem which is easy to solve in Matlab. The resulting 
coefficients are given in Tables~\ref{tab:coeff_AC_2D} and 
\ref{tab:coeff_AC_3D} below. 
The maximum over angle of the error
$ \left. \frac{ Q(\xi) }{ \sqrt{ 
    \frac{\| \partial P/\partial\xi(\xi) \| }
        { \| \partial H/\partial \xi (\xi) \| } } } \right|^2_{\xi = \omega s_{\rm ph}(\theta) \theta}
$ varied between around $10^{-8}$ for $1/G =0.05$ and
$10^{-2}$ for 1/G = 0.4.

\begin{table}
\begin{center}
\small
\begin{tabular}{c|cccc}
$1/G$ 
& $\beta_1$ & $\pdpd{\beta_1}{(1/G)}$
& $\beta_2$ & $\pdpd{\beta_2}{(1/G)}$\\ \hline
0.00 &  0.872589 & -0.115476 &  0.088139 &  0.232493 \\ 
0.05 &  0.870989 & -0.080799 &  0.089351 &  0.080994 \\ 
0.10 &  0.866560 & -0.122182 &  0.092018 &  0.075452 \\ 
0.15 &  0.858994 & -0.189920 &  0.096178 &  0.106183 \\ 
0.20 &  0.847495 & -0.277477 &  0.102309 &  0.147420 \\ 
0.25 &  0.830913 & -0.394429 &  0.110797 &  0.198380 \\ 
0.30 &  0.807375 & -0.559277 &  0.122158 &  0.261263 \\ 
0.35 &  0.773715 & -0.806746 &  0.137030 &  0.337561 \\ 
0.40 &  0.724163 & -1.211119 &  0.155971 &  0.420753 \\ 
\hline
\end{tabular}%
\end{center}
\caption{Coefficients amplitude correction operator  $Q$ in 2-D}
\label{tab:coeff_AC_2D}
\end{table}
\begin{table}
\begin{center}\small
\hspace*{-2cm}%
\begin{tabular}{c|cccccc}
$1/G$ 
& $\beta_1$ & $\pdpd{\beta_1}{(1/G)}$
& $\beta_2$ & $\pdpd{\beta_2}{(1/G)}$
& $\beta_3$ & $\pdpd{\beta_3}{(1/G)}$ \\ \hline
0.0000 &  0.806683 &  0.002423 &  0.193113 & -0.002685 & -0.056266 & -0.002551 \\
0.0500 &  0.832963 & -0.081724 &  0.114016 &  0.032813 &  0.020075 &  0.058590 \\
0.1000 &  0.841034 & -0.130484 &  0.076623 &  0.029868 &  0.061360 &  0.078398 \\
0.1500 &  0.833587 & -0.231333 &  0.076280 &  0.129614 &  0.067935 &  0.024410 \\
0.2000 &  0.821230 & -0.304691 &  0.078943 &  0.086321 &  0.074389 &  0.130587 \\
0.2500 &  0.803736 & -0.416375 &  0.081855 &  0.072002 &  0.084073 &  0.220607 \\
0.3000 &  0.779384 & -0.573760 &  0.084646 &  0.054207 &  0.098065 &  0.329810 \\
0.3500 &  0.745468 & -0.801027 &  0.086156 &  0.004734 &  0.118341 &  0.486328 \\
0.4000 &  0.697405 & -1.148951 &  0.083351 & -0.136764 &  0.148391 &  0.732785 \\
\hline
\end{tabular}%
\hspace*{-2cm}
\end{center}
\caption{Coefficients amplitude correction operator  $Q$ in 3-D}
\label{tab:coeff_AC_3D}
\end{table}

\subsection{Comparison of phase slowness errors\label{eq:impact_improvements}}
The following conclusions can be drawn from the data in
Figures~\ref{fig:phase_errors_exist_2D},
\ref{fig:phase_errors_exist_3D} and~\ref{fig:IOFD_phase_errors}.
First the QS-FEM method of \cite{BabuskaEtAl1995} (in two dimensions)
and the IOFD method developed here (in two and three dimensions)
perform remarkably well considering their small stencils. They
provides a substantial improvement, roughly a factor 20, in
phase errors compared to the compact sixth order scheme SUT and CHO6 
of \cite{Sutmann2007} and \cite{TurkelEtAl2013}, which in turn are better than other
alternatives. For higher order FD and FE methods, as can be expected,
the error becomes small if both the number of points per wavelength
and the order $N$ become large, however this effect sets in quite
late, e.g.\ at eight points per wavelength and $N=8$ the relative
phase slowness errors of the finite element method are roughly equal to
those of QS-FEM and IOFD.
 
Next we discuss how much accuracy might be needed, and in how far the
improvements will make a difference in simulations.
In view of (\ref{eq:phase_error_from_dispersion_error})
it is not unreasonable to require at least that
$\delta_{\rm ph} \lesssim 0.01 \frac{\lambda}{L}$.  In a regime of wave
propagation over several hundreds wavelengths, using a mesh with five
points per wavelength, from the methods considered only QS-FEM and
IOFD satisfy this. At six points per wavelength the CHO6 method is
near this bound while FE8 (which is much more expensive) also
qualifies. So in these situations the improved phase slowness accuracy
obtained by using QS-FEM or IOFD can be expected to have some impact
in terms of lower cost compared to FE8 and in terms of improved 
accuracy compared to CHO6 and other compact finite difference methods. The 
latter will be confirmed in the examples in the next section.

\section{Numerical examples}
\label{sec:simulations}

In this section we present two numerical experiments, first in a
constant medium, and then in a smoothly varying medium. We will
present two-dimensional examples with large domain sizes on the order
of hundreds of wavelengths.

As mentioned, phase slowness errors typically lead to phase shift
errors in the solutions. Considering wave propagation over 500
wavelengths as an example, it follows from
(\ref{eq:phase_error_from_dispersion_error}) and the surrounding
discussion that these phase shifts errors for IOFD should be
negligibly small for meshes with five or six points per wavelength,
and still quite small for four and three points per wavelength. For
other methods these errors should show up much stronger. In our first
example we will verify this numerically, assuming a constant velocity
model.
 
To simulate a point source at a given grid point, we will simply use a
discrete $\delta$-function. An unbounded domain is simulated by adding
a damping layer around the domain of interest, with a nonzero
imaginary contribution to $k$ that quadratically increases from the
boundary of the domain of interest%
\footnote{
In a 1-D damped Helmholtz equation $-\frac{du}{dx^2} - k^2 u$ with
$k$ constant, $k = \alpha + i \beta$ solutions decay as 
$u = e^{i (\alpha + i \beta)x}$. If $k$ varies slowly the damping
  becomes proportional to $e^{ - \int \beta(x) \, dx}$. The quadratic
  profile is chosen such that $e^{ - \int \beta(x) \, dx}$ is on the
  order of $0.001$ to $0.01$. Reflected waves pass twice through 
  the damping layer. Unfortunately reflections occur due to the medium
  variations. To make these small $\operatorname{Im}(k)$ must increase
  slowly and these layers must be quite thick. In our experiments we 
  used on the order of 5 to 10 wavelengths.}. 
The discrete system of equations
is formed using a Matlab code written for this purpose and then either
solved directly or, for the larger examples, exported to disk. In the
latter case, the resulting linear systems are then solved using the
MUMPS parallel direct solver \cite{MUMPS:1}  on a few nodes of the Lisa
cluster of surfsara (www.surfsara.nl). This
system contains 32 parallel nodes with each two intel Xeon
processors E5-2650 v2 running at 2.60 GHz and 64 GB memory, connected
by Mellanox FDR Infiniband. In the examples in of section between
1 and 4 nodes were used in parallel.

To easily observe the absence or presence of the phase shifts, we plot
the resulting wave field on a 45 degree segment of an annulus, with
the radial coordinate varying on an interval of about a
wavelength. The location where the real part is minimal, according to
the exact solution, is indicated by a line that is plotted. The
transformation of the field to polar coordinates is done by using
cubic interpolation from the numerical solution on a Cartesian
mesh. Schematically this is displayed in
Figure~\ref{fig:const_example_schematic}, where part (b) of the figure
is a plot in polar coordinates of the indicated region of part (a).
\begin{figure}
\begin{center}
(a)\\
\includegraphics[width=45mm]{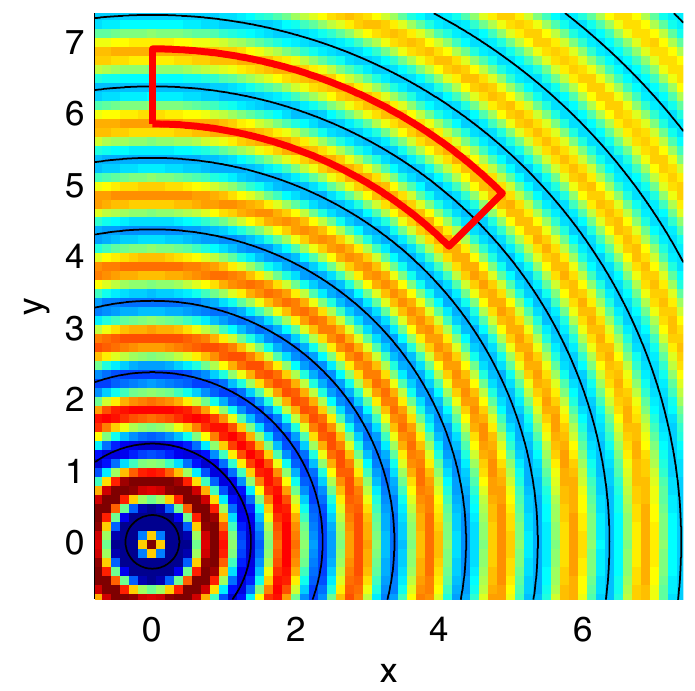}\\
(b)\\
\includegraphics[width=45mm]{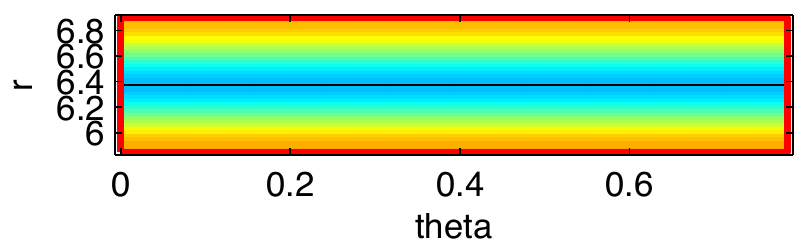}\\
\end{center}
\caption{Phase shift errors are easily observed by plotting a 45-degree part
of an annulus, see figure (b). Cubic spline interpolation is applied
to map the data of figure (a) to polar coordinates.}
\label{fig:const_example_schematic}
\end{figure}

The results from the computations are displayed in
Figure~\ref{fig:const_example}. Part (a) shows that for second order
finite differences at 10 points per wavelength (ppw)
a clearly visible phase shift already
occurs after 20 wavelengths. In (b) we see that for the JSS method
a clearly visible phase shift occurs after 50 wavelengths. 
In (c), (d) and (e) we investigate the
sixth order method CHO6 of \cite{TurkelEtAl2013} at 6, 5 and 4 ppw. 
(We have chosen one of the higher order methods).
At 6, 5 and ppw the maximum phase errors at 500 wavelengths are 0.27,
0.84 and $\pi$ radians respectively and the associated phase shifts
are increasingly visible in the pictures.
In parts (e) to (i) we plot results for
the IOFD method at 6, 5, 4, 3 and 2.5 ppw. 
At 6, 5 and 4 ppw the maximum phase shifts are 0.0065, 0.020 and
0.089 radians respectively, i.e.\ considerably smaller than observed for
CHO6. At 3 ppw the phase
shift after 500 wavelengths is clearly visible, only at 2.5 points per
wavelength does it become large and in this case the field is plotted
at 100 instead of less than 500 wavelengths from the source.

In Figure~\ref{fig:const_ampl_example} we plot the amplitude errors for
IOFD at 3 and 4 ppw. For more than 4 ppw they were increasingly small.

\begin{figure}
\begin{center}
\begin{minipage}[t]{40mm}
\begin{center}
(a)\\
\includegraphics[width=35mm]{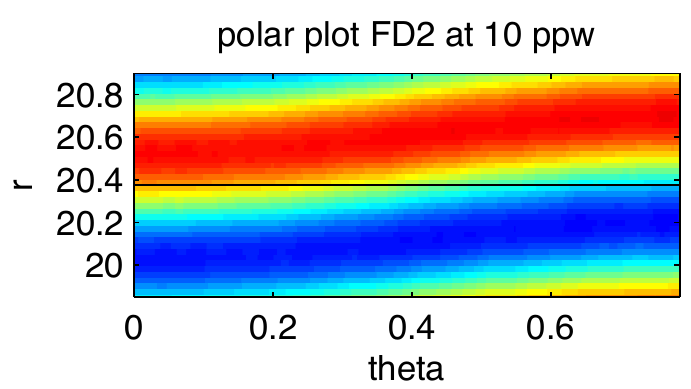}\\
(b)\\
\includegraphics[width=35mm]{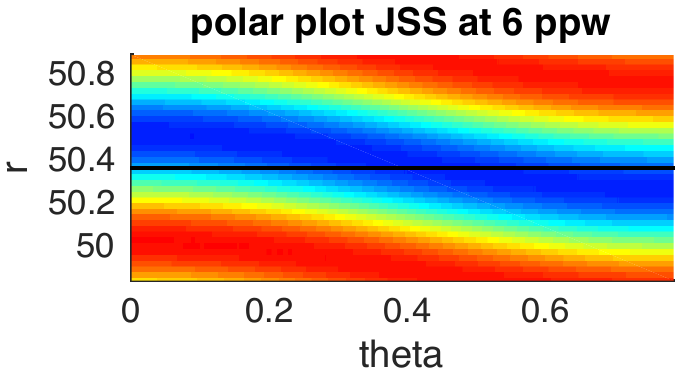}\\
(c)\\
\includegraphics[width=37mm]{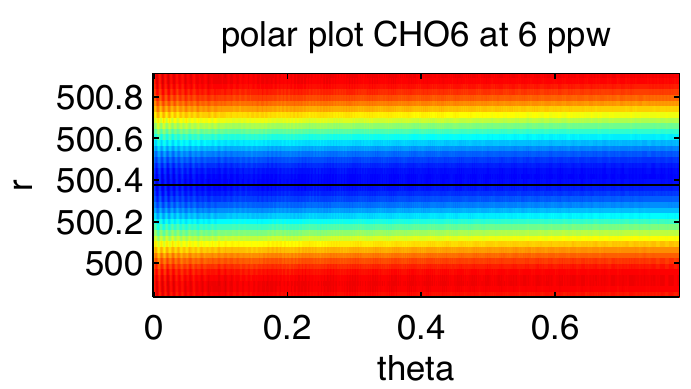}\\
(d)\\
\includegraphics[width=37mm]{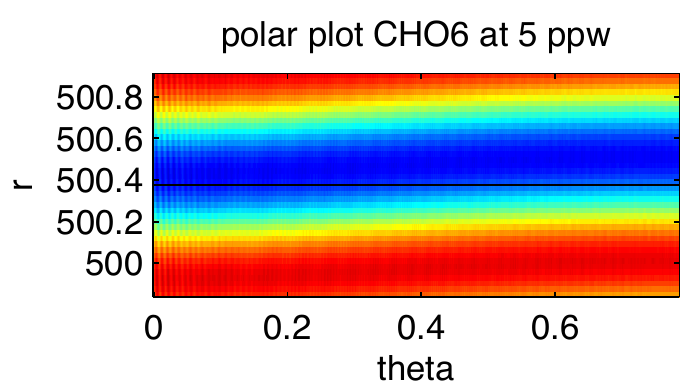}\\
(e)\\
\includegraphics[width=37mm]{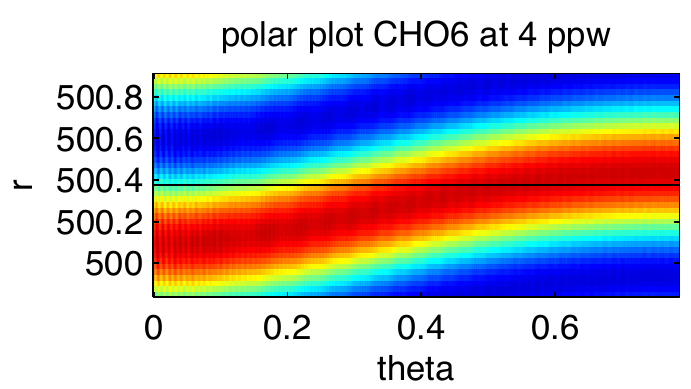}\\
\end{center}
\end{minipage}
\hspace*{10mm}
\begin{minipage}[t]{40mm}
\begin{center}
(f)\\
\includegraphics[width=37mm]{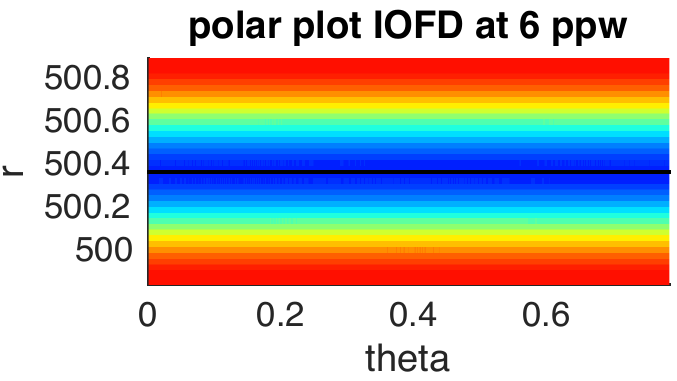}\\
(g)\\
\includegraphics[width=37mm]{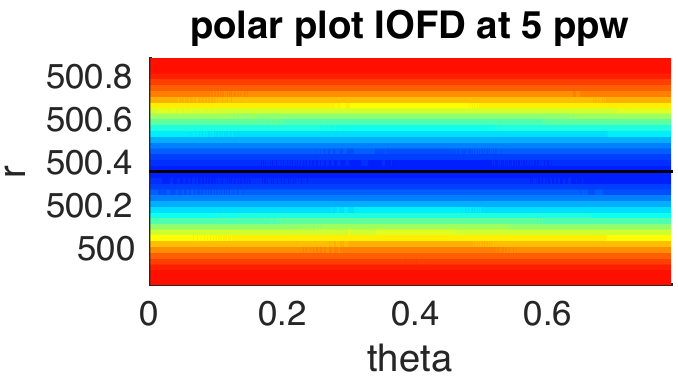}\\
(h)\\
\includegraphics[width=37mm]{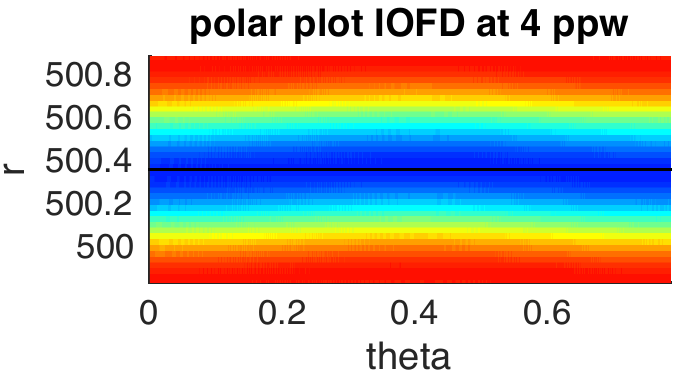}\\
(i)\\
\includegraphics[width=37mm]{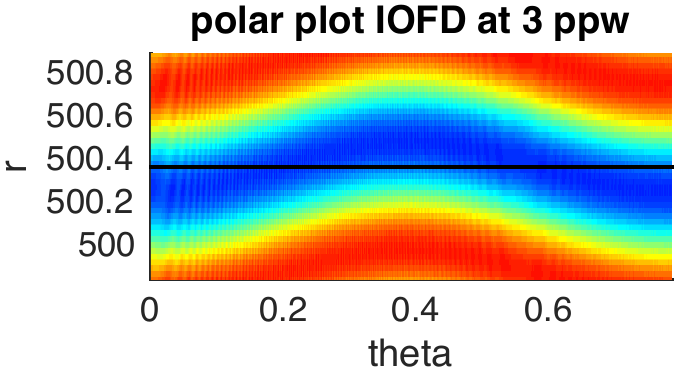}\\
(j)\\
\includegraphics[width=37mm]{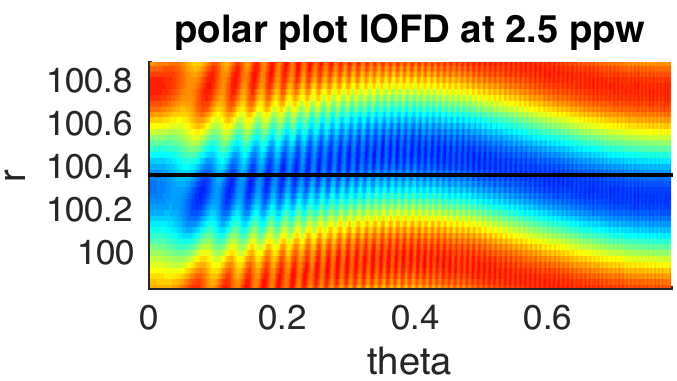}\\
\end{center}
\end{minipage}
\end{center}
\caption{Plots of numerical solutions over a 45 degree part of an
  annulus for several numerical methods. (a) FD2 at 10 ppw; (b) JSS at
  6 ppw; (c), (d), (e) sixth order method of \cite{TurkelEtAl2013} at 6, 5 and 4 ppw;
  (f)-(j) IOFD method at 6, 5, 4, 3 and 2.5 ppw.}
\label{fig:const_example}
\end{figure}

\begin{figure}
\begin{center}
\begin{minipage}[t]{40mm}
\begin{center}
(a)\\
\includegraphics[width=37mm]{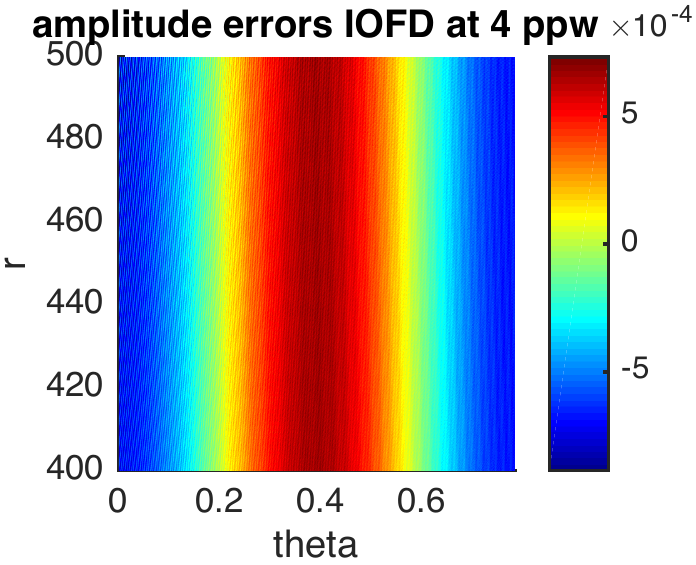}\\
\end{center}
\end{minipage}
\hspace*{10mm}
\begin{minipage}[t]{40mm}
\begin{center}
(b)\\
\includegraphics[width=37mm]{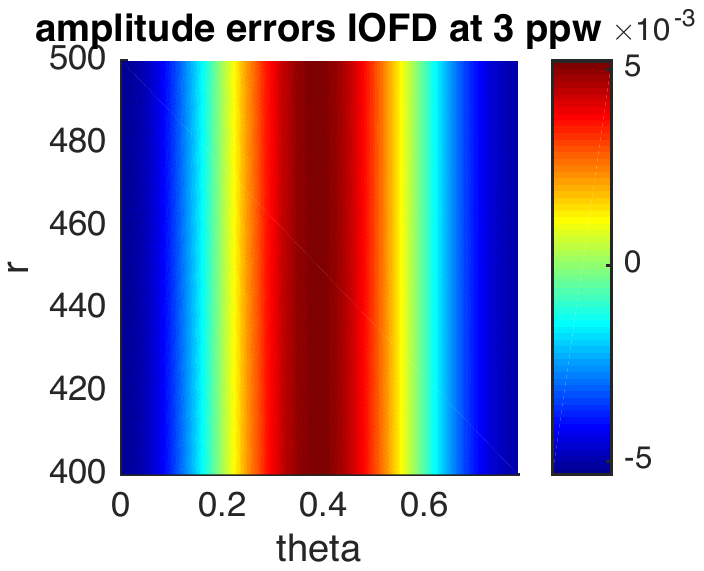}\\
\end{center}
\end{minipage}
\end{center}
\caption{Amplitude errors relative to 
exact solution for IOFD at (a) 4 ppw (b) 3 ppw.}
\label{fig:const_ampl_example}
\end{figure}

In our second example $k$ is variable. To avoid that errors due to the
discretization of the velocity model become dominant we use use a
smoothly varying velocity model, namely a smoothed Marmousi model. In
this example we will compare a solution with IOFD using a minimum of
six points per wavelength with a fourth order finite element solution
using twice as many grid points in each direction. In these examples
the right hand side was a point source and the linear systems were
again solved with MUMPS.  In case of variable coefficients we assumed
that $k(x)$ is defined on the cell centers. The values of
$f_\gamma(kh(x))$ (see (\ref{eq:define_P_variable_k})) at other points
were obtained using linear interpolation from the values of
$f_\gamma(hk(x))$ at the cell centers.

The velocity model is given in Figure~\ref{fig:marmousmooth_vel}. It
is obtained from the Marmousi model by convolving along both of the
axes with a cos square pulse of width 160 meter. We will give results
for 50 and 100 Hz. A solution for the first case is given in
Figure~\ref{fig:marmousmooth_sol}. Figure~\ref{fig:marmousmooth_compare}
contains four plots. The top plots are reference amplitudes for 
obtaining relative errors and the bottom two plots are relative errors
with respect to the reference values. In both cases we give results
for 50 and for 100 Hz. The reference value is a local
average of the absolute value of the solution over a square of about 2
by 2 wavelengths. This is done because the solutions themselves
contain nodal points from interfering waves, where the
amplitude is very small, and are hence not directly suitable
as reference value. Very small relative errors
are obtained (except directly at the source point), ranging from less
than 0.01 over most of the domain to 0.05 or 0.08 at isolated spots 
where the absolute amplitude is small.

\begin{figure}
\begin{center}
\includegraphics[width=16cm]{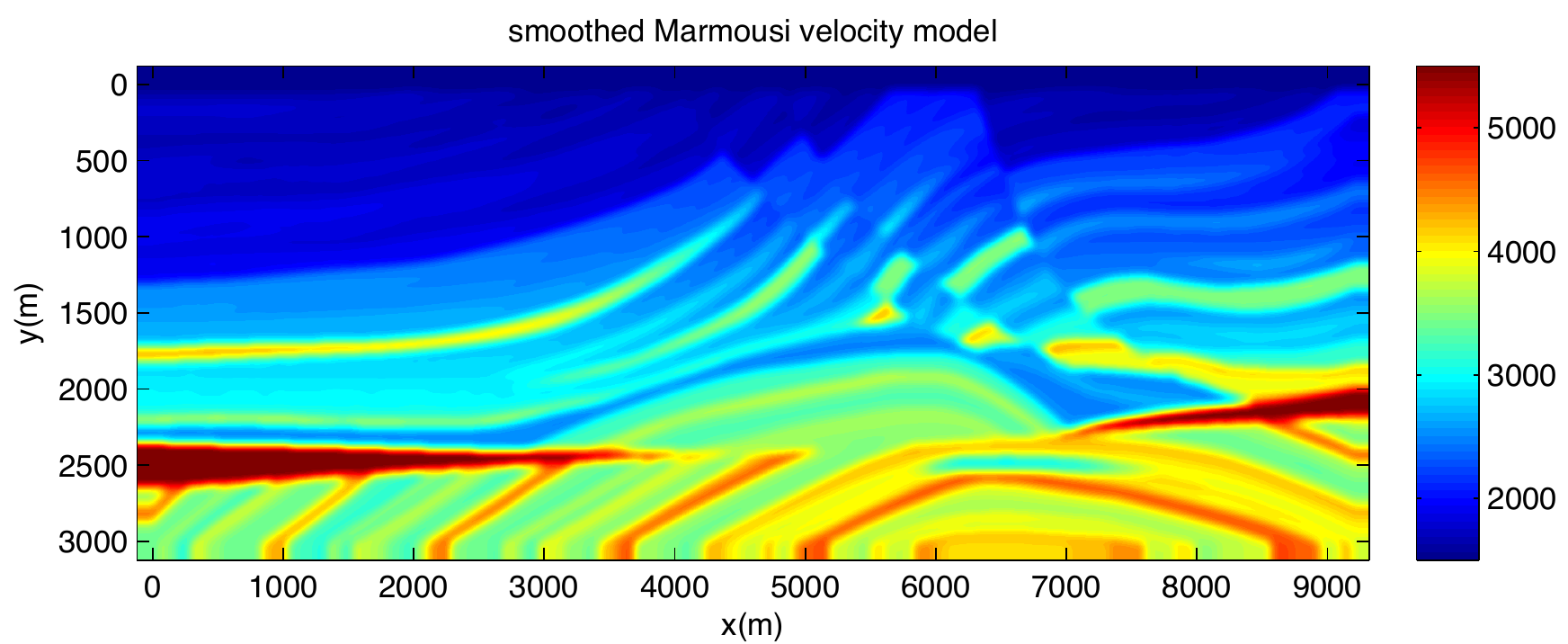}\\
\end{center}
\caption{Smoothed Marmousi velocity model}
\label{fig:marmousmooth_vel}
\end{figure}
\begin{figure}
\begin{center}
\includegraphics[width=16cm]{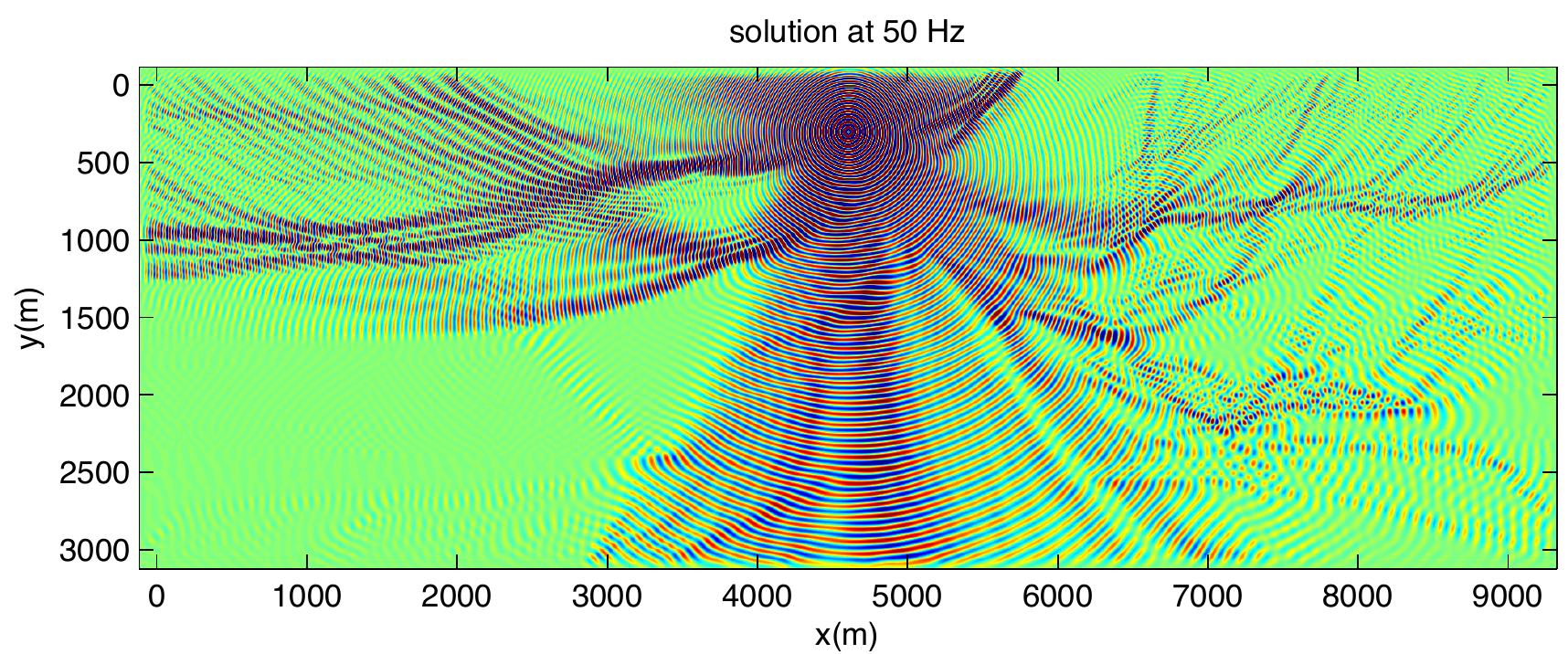}\\
\end{center}
\caption{Solution from a point source at 50 Hz}
\label{fig:marmousmooth_sol}
\end{figure}

\begin{figure}
\hspace*{-10mm}
\begin{minipage}{170mm}
\begin{center}
(a) \hspace*{78mm} (b)
\\
\includegraphics[width=80mm]{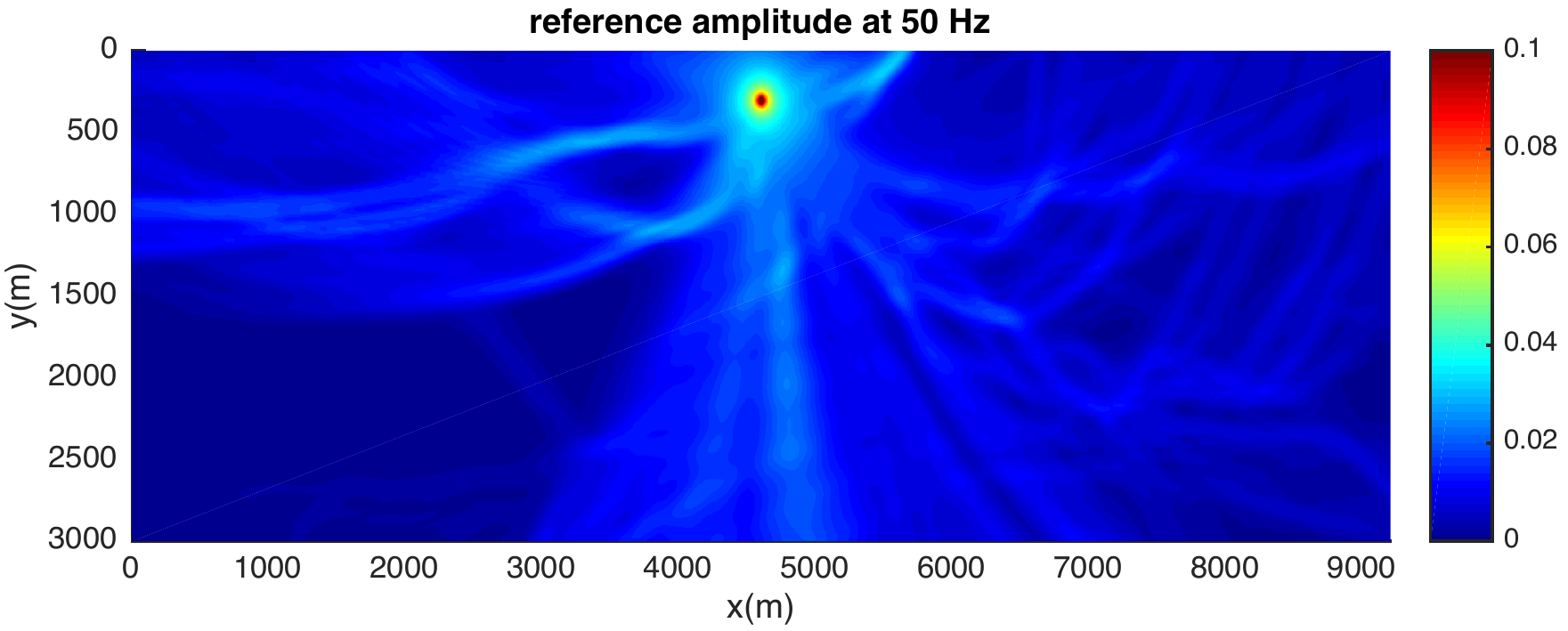}%
\includegraphics[width=80mm]{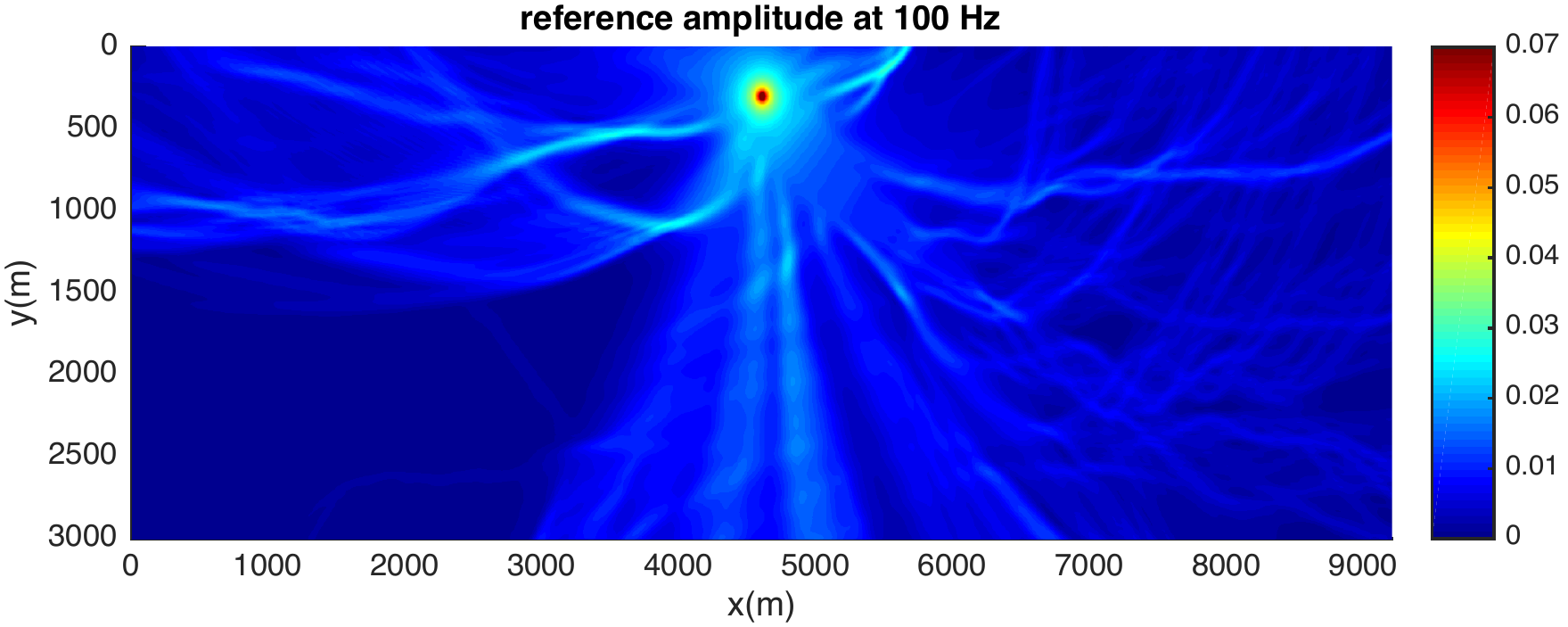}
\\
(c) \hspace*{78mm} (d)
\\
\includegraphics[width=8cm]{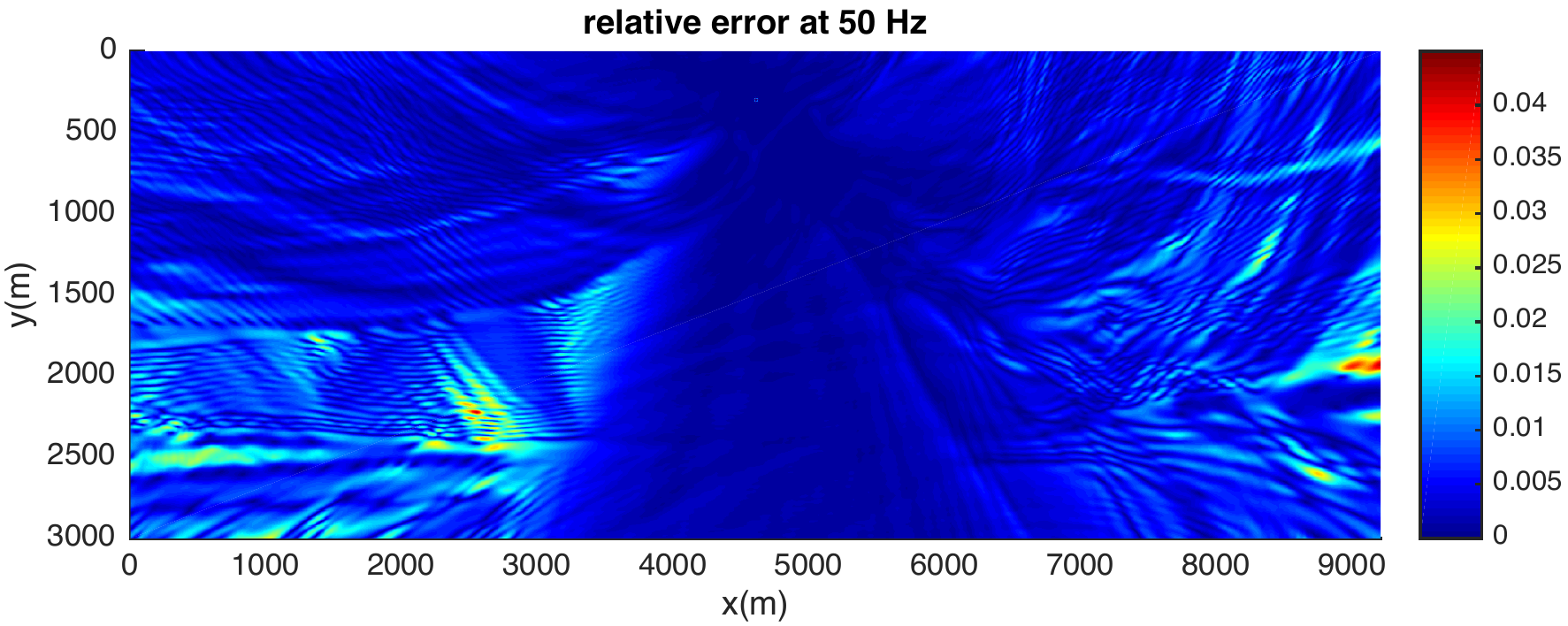}%
\includegraphics[width=8cm]{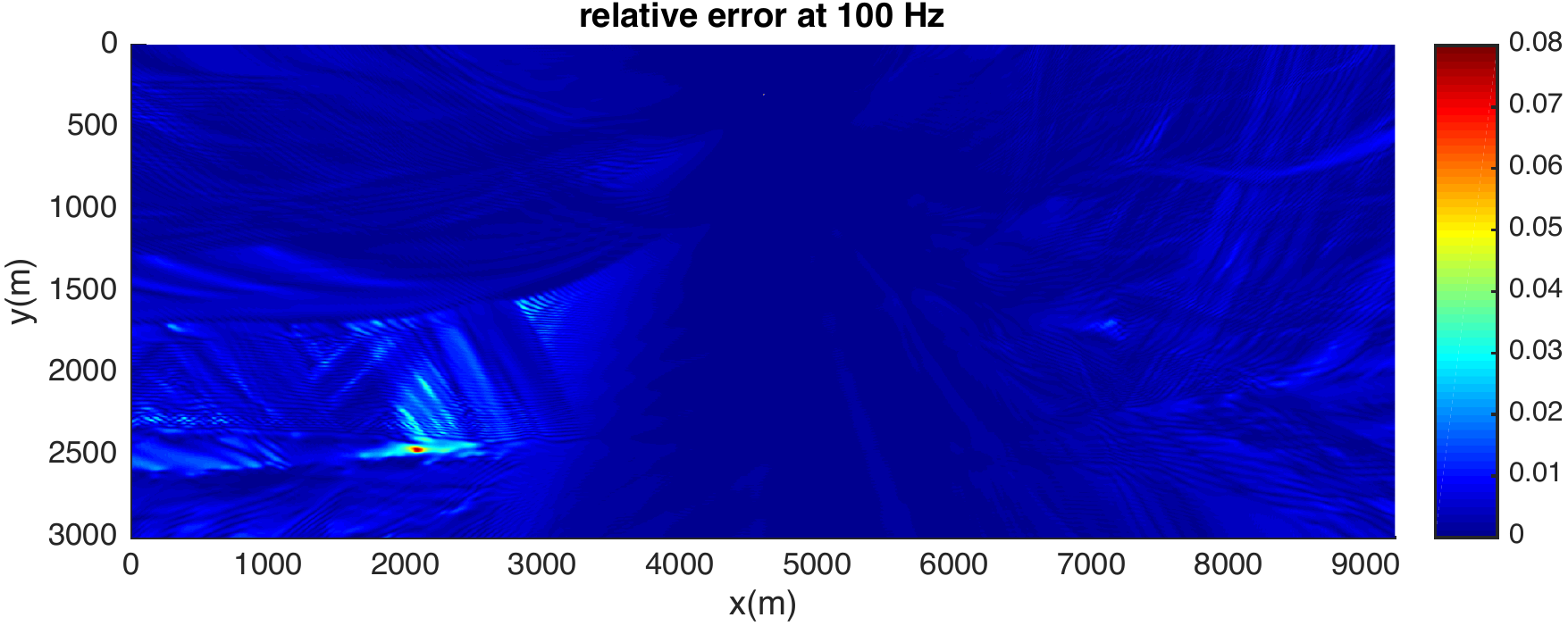}
\end{center}
\end{minipage}
\caption{Reference values for 50 Hz (a) and for 100 Hz (b). Relative
  errors in the solutions for 50 Hz (c) and for 100 Hz (d).}
\label{fig:marmousmooth_compare}
\end{figure}

\section{Application in multigrid based solvers}
\label{sec:twogrid}

The last few years there have been several interesting developments in
multigrid methods for Helmholtz equations. Different two-grid methods with
inexact coarse level solvers have been studied in
\cite{CalandraEtAl2013} and \cite{Stolk2014Preprint}. In
\cite{CalandraEtAl2013} a number of iterations of shifted Laplacian
preconditioned Krylov solver \cite{ErlanggaOosterleeVuik2006} is used
as coarse level solver. The method of \cite{Stolk2014Preprint} is
based on the multigrid method in \cite{StolkEtAl2014} with a double
sweep domain decomposition preconditioner \cite{Stolk2013} as coarse
level solver.  The multigrid method with exact coarse level solver was
studied in \cite{StolkEtAl2014}.  There it was shown that the
convergence can be strongly improved when phase slowness differences
between the fine and coarse scale operators are minimized. For this
purpose, optimized finite differences were used at the coarse level,
and good convergence was obtained for meshes with downto three points
per wavelength at the coarse level. For standard choices of the coarse
level discretization it was found that about 10 points per wavelength 
at the coarse level were needed to have good convergence.

In \cite{StolkEtAl2014} standard second order finite differences were
used as the fine level. Because of the relatively large phase slowness errors
of this method, the coarse level optimized finite difference method
had to be constructed specifically to match the phase slownesss of
second order finite differences, instead of matching the true
phase slowness. A better choice is to use method with small
phase slowness errors at the fine level and at the coarse levels. Here we
will use IOFD at all levels.
These experiments do not involve the operator $Q$. The operator $P$ is
used directly as coarse level discretization and at the fine level we
are only interested in solving equation
(\ref{eq:discrete_system_abstract1}).

In the first set of computational results of this section we will show
that this results in good convergence of the multigrid method with
exact coarse level solver. In a second example we will study the
multigrid method with inexact coarse level solver of
\cite{Stolk2014Preprint}. 

In our study of the convergence when using the exact coarse level
solver we are again interested in examples with wave propagation over
hundreds of wavelengths. Therefore, these experiments are done in two
dimensions. 
For background on multigrid methods, see
\cite{TrottenbergOosterleeSchueller2001}. 
As in \cite{StolkEtAl2014} most of the components of the multigrid
method are standard. Full weighting restriction and prolongation
operators are used. As smoother, an $\omega$-Jacobi method is used. We
found that  $\omega = 0.7$ and $\nu = 4$ (the number of pre- and
postsmoothing parameters) are good choices of parameters. 
In these experiments we used a conventional absorbing boundary layer
to simulate an unbounded domain. 

We studied the convergence as a function of the number of points per
wavelength for three velocity models: A constant
model, the Marmousi model and a slice of the 3-D SEG-EAGE salt
model. The latter two models are displayed in
Figure~\ref{fig:twogrid_models}.  The parameters of the examples and
the observed number of iterations to reduce the residual by $10^{6}$
are given in \ref{tab:two-grid}.  At downto three points per
wavelength the method behaved well. At 2.5 ppw coarse level the method
still converged, but the number of iterations increased substantially,
and also became more sensitive to the problem size (which was apparent
from smaller scale experiments not included in the table).

Note that the application in multigrid methods is quite different from
the application as fine level discretization.  The method is used at
coarser meshes (at three points per wavelength the direct application
will in general lead to too large errors). Also, multigrid
solvers using IOFD at coarse levels may be developed for other types
of fine level discretizations, as long as they use a regular mesh.

\begin{figure}
\begin{center}
(a)\\
\includegraphics[width=72mm]{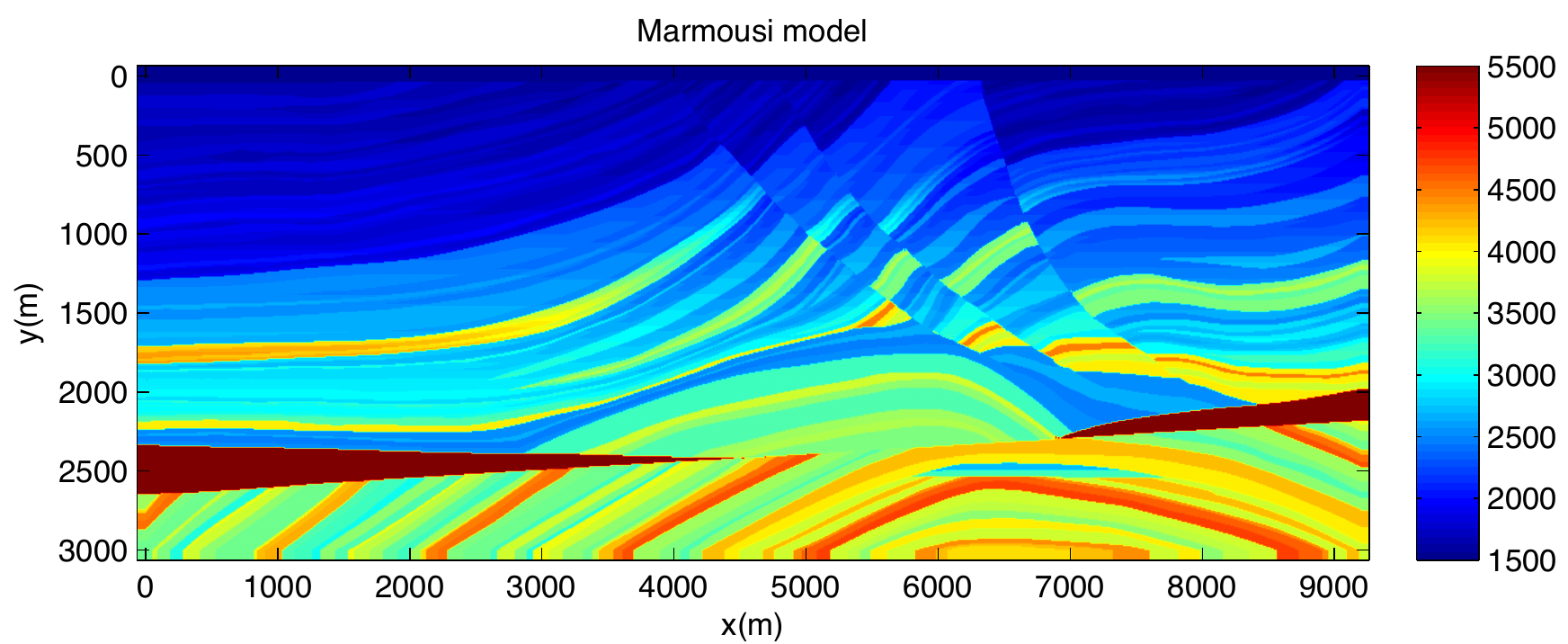}\\
(b)\\
\includegraphics[width=72mm]{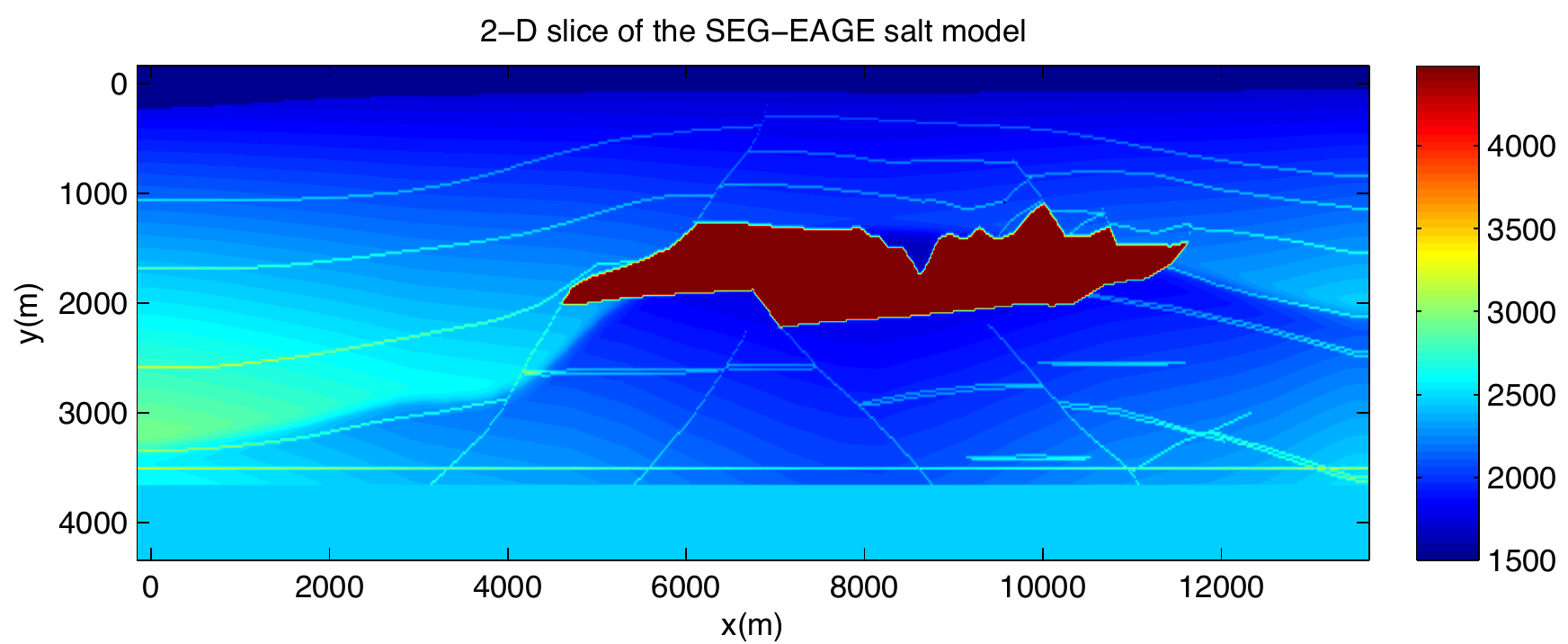}
\end{center}
\caption{Velocity models for the 2-D two-grid experiments. (a)
  Marmousi (b) 2-D slice of the SEG-EAGE salt model.}
\label{fig:twogrid_models}
\end{figure}

\begin{table}
\begin{center}
\begin{tabular}{c|cc|cc|cc} \hline
  & \multicolumn{2}{c|}{constant} 
  & \multicolumn{2}{c|}{Marmousi}
  & \multicolumn{2}{c}{salt model} \\ 
  & \multicolumn{2}{c|}{$2400 \times 2400$}
  & \multicolumn{2}{c|}{$4600 \times 750$}
  & \multicolumn{2}{c}{$2700 \times 836$} \\ \hline
ppw & freq & its  & freq & its & freq & its \\ \hline
5  & 480   & 29 & 150   & 23 & 60   & 18\\
6  & 400   & 8  & 125   & 11 & 50   & 8 \\
7  & 342.9 & 6  & 107.1 &  9 & 42.9 & 7 \\
8  & 300   & 5  & 93.8  &  8 & 37.5 & 6 \\
9  & 266.7 & 5  & 83.3  &  7 & 33.3 & 6 \\
10 & 240   & 4  & 75    &  6 & 30   & 5 \\ \hline
\end{tabular}
\end{center}
\caption{Iterations required for a two-grid method using IOFD
  discretization at the fine and coarse level as a function of the
  number of points per wavelength (ppw)}
\label{tab:two-grid}
\end{table}

We now turn to a multigrid method with an inexact coarse level solver.
Such methods are used because in three dimensions it is often too
expensive to compute the exact solution. These methods are currently
some of the fastest solvers for large problems that are in the
literature \cite{CalandraEtAl2013,Stolk2014Preprint}.

Because we are interested in coarse meshes, such as six points per
wavelength based on the previous examples, it is {\em a priori} not
clear that the above mentioned solvers perform well. Like many solvers
in the literature, they were tested for problems with at least ten
mesh points per wavelength. They cannot be assumed to converge as well
for larger frequencies, because multigrid convergence depends on
frequency, and the same is true for the shifted Laplacian
preconditioner \cite{CoolsVanroose2013}. For the double sweep domain
decomposition it is unclear how the frequency affects the convergence,
but a priori it also cannot be assumed to be independent of the
frequency.

This raises the question whether we can actually obtain a gain in
efficiency by going to coarser meshes. The purpose of the next example
is to show that this indeed the case, and to generally show that IOFD
can perform well with the solver of \cite{Stolk2014Preprint}.

In the following example we will test the method of
\cite{Stolk2014Preprint}, which is a two-grid method using an inexact
coarse level solver given by a double sweep domain decomposition
preconditioner (see \cite{Stolk2013}). The method is modified to use
IOFD at both the fine and the coarse levels of the two-grid method.
We will take the SEG-EAGE Salt Model as an example, similarly as in 
\cite{Stolk2014Preprint}. In addition to changing the discretization
method we will increase the frequency by a factor $\frac{5}{3}$,
so that a minimum of six points per wavelength is used, a regime which
has not been tested before for this method. If convergence and cost 
per degree of freedom would stay constant, there would be an
improvement in the cost by a factor of over $\big( \frac{5}{3} \big)^3
\approx 4.62$ (more than this because cost grows somewhat faster than
linear with problem size).

The original SEG-EAGE salt model is of size 13500 x 13500 x 4200
meter, discretized with 20 m grid spacing. We apply the method just
described to solve the Helmholtz equation with this velocity model and
random or point sources as right hand sides at four different
frequencies from $6.25$ to $12.5$ Hz.  Slices of the model are
displayed in Figure~\ref{fig:saltmodel_model}.  Parameters in the
two-grid method are $\nu = 3$ for the number of pre- and postsmoothing
steps and $\omega = 0.65$ in the $\omega$-Jacobi method.
Computations were done on the Lisa cluster at Surfsara, described
already in section~\ref{sec:simulations},
using the implementation described in~\cite{Stolk2014Preprint}. A maximum of 16 nodes were used in parallel for
these computations. 

The algorithm is set up to solve for multiple right hand sides
simultaneously. In the table of results, the computation time per
right hand side is given. In Table~\ref{tab:SaltModel} some parameters
are given, together with the computation time and iteration count to
reduce the residual by $10^{-6}$. 
As illustration, plots of a solution are given in
Figure~\ref{fig:saltmodel_solution}.
It can be observed that the cost increases very little compared to
the results of \cite{Stolk2014Preprint}, even though frequencies are increased
by a factor $5/3$. Some increase in cost can be expected, because the
discrete Helmholtz operator using second order finite differences is 
cheaper to apply than the one using a compact 27-point stencil.
Hence reducing the number of points per wavelength in the mesh can
indeed lead to corresponding savings in computation time.

As mentioned, the methods of \cite{CalandraEtAl2013} and
\cite{Stolk2014Preprint} are some of the fastest currently in the
literature. Comparing with these results we see a significant
improvement. For example in \cite{CalandraEtAl2013} the SEG-EAGE salt
model problem was solved at 10 Hz in 270 seconds on 256 cores of an IBM
BG/P machine (with the residual reduced by a factor $10^{5}$ instead
of $10^6$ in our case). Here we solve the problem at 9.91 Hz using 128
cores in 45 seconds per right hand sides (179 seconds for four right
hand sides), a clear improvement\footnote{%
On the other hand the method of \cite{CalandraEtAl2013} uses less 
memory and has been applied to larger examples than we have shown 
here. Furthermore no full comparison including accuracy was made.}.

\begin{figure}
\begin{center}
(a) \hspace*{6cm} (b)\\
\includegraphics[width=7cm]{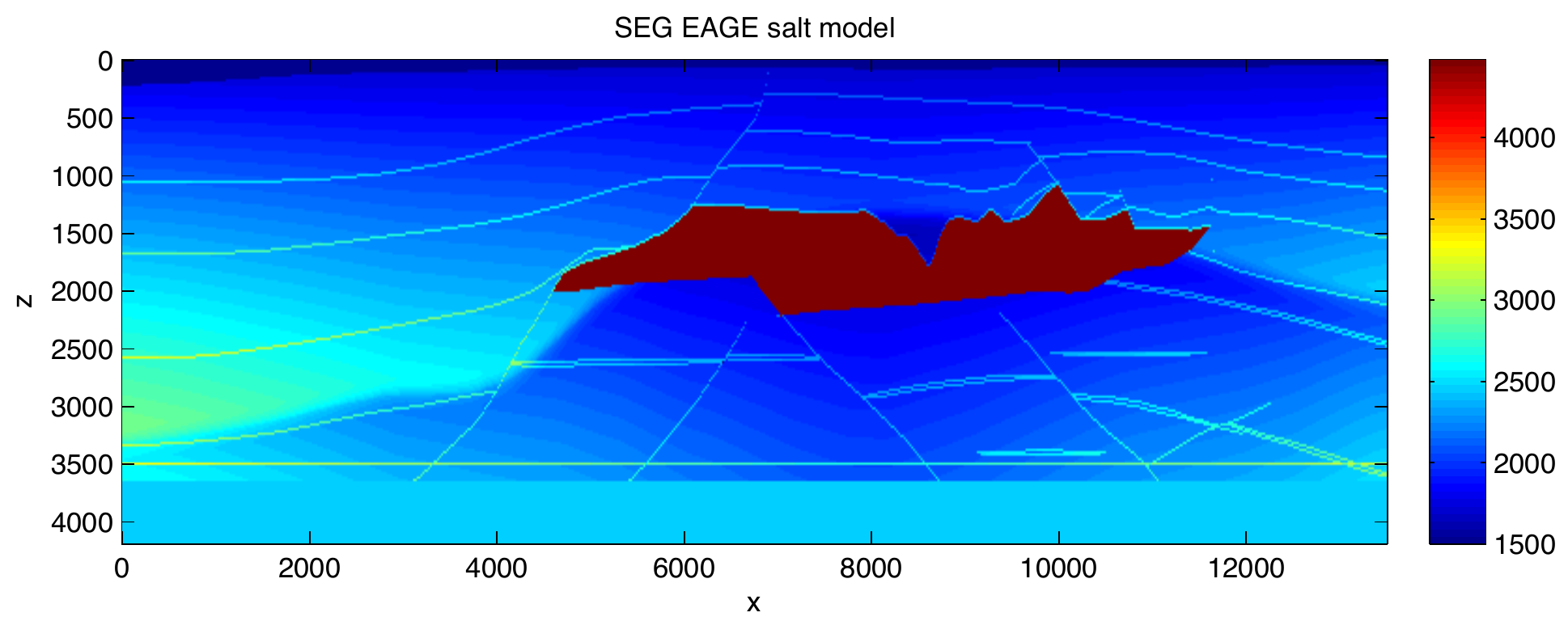}
\includegraphics[width=7cm]{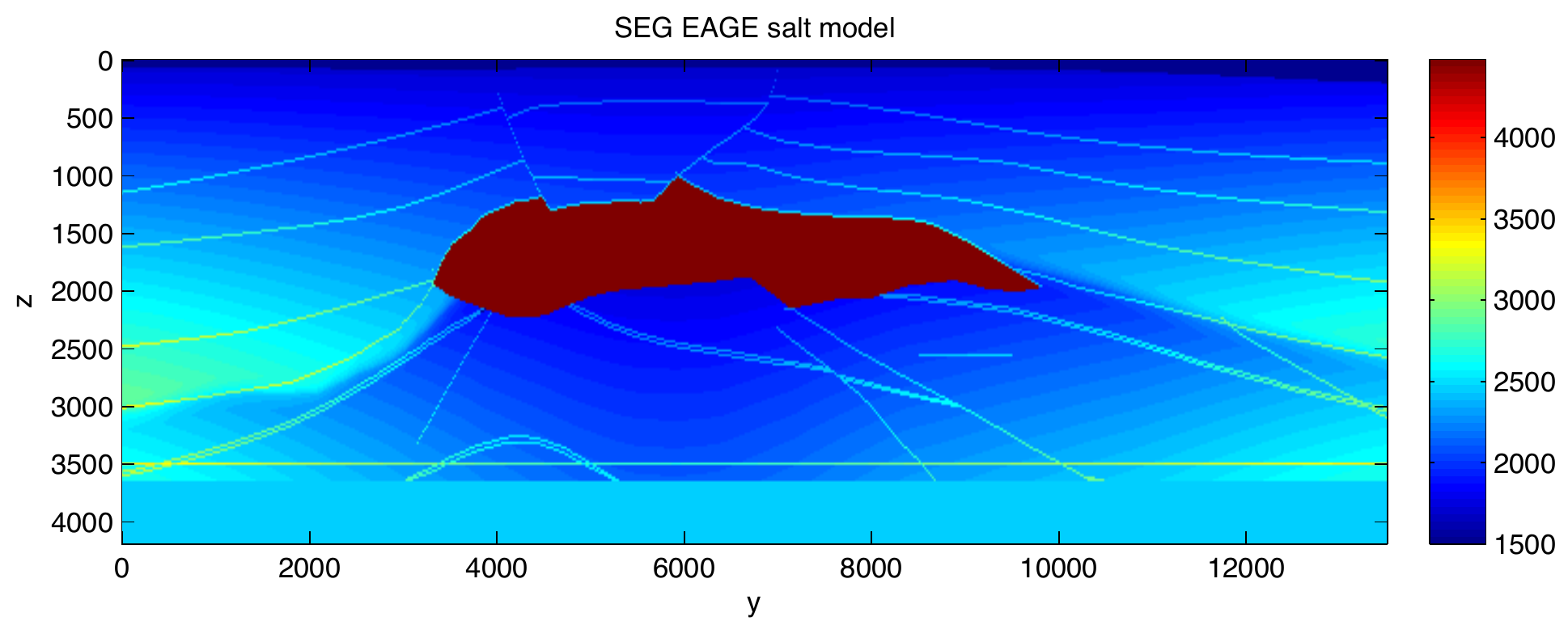}
\end{center}
\caption{SEG-EAGE salt velocity model: (a) $(x,z)$ slice at $y=6740$ m 
(b) $(y,z)$ slice at $x=6740$ m.}
\label{fig:saltmodel_model}
\end{figure}
\begin{figure}
\begin{center}
(a) \hspace*{7cm} (b)\\
\hspace*{-1cm}%
\includegraphics[width=7cm]{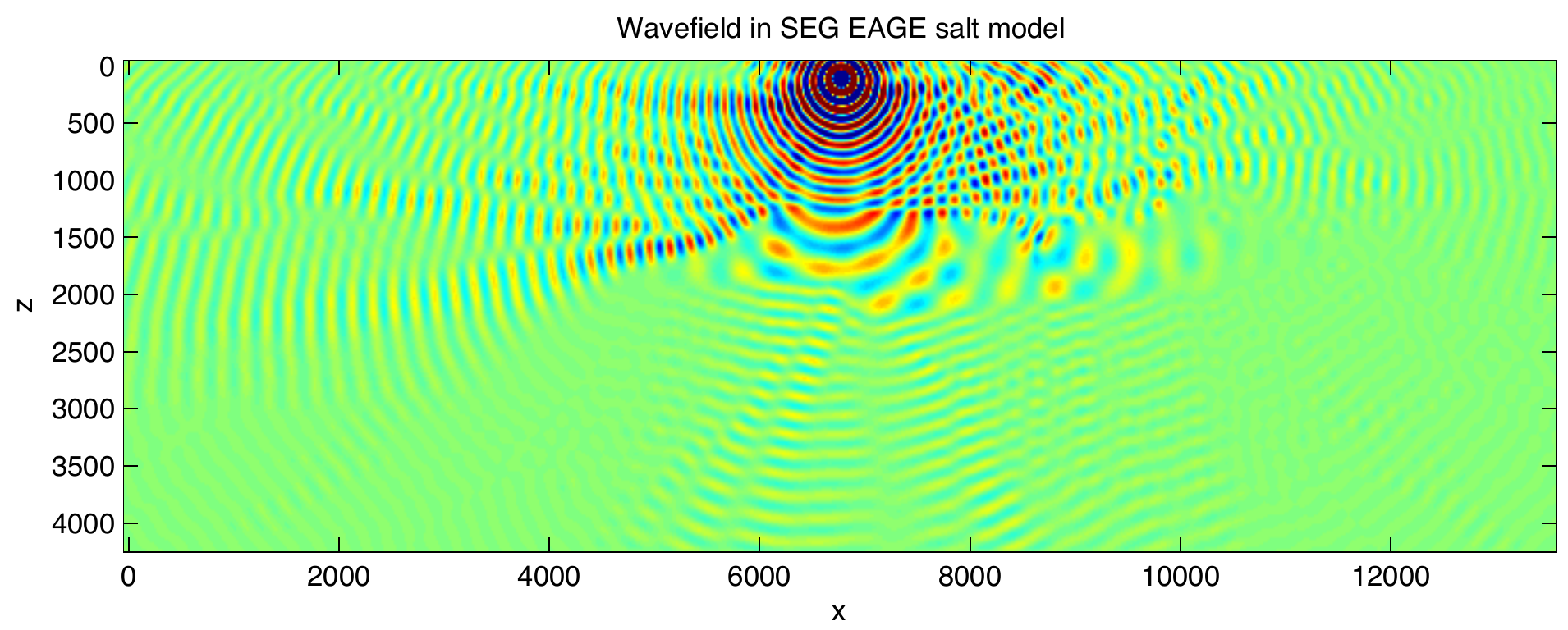}%
\includegraphics[width=7cm]{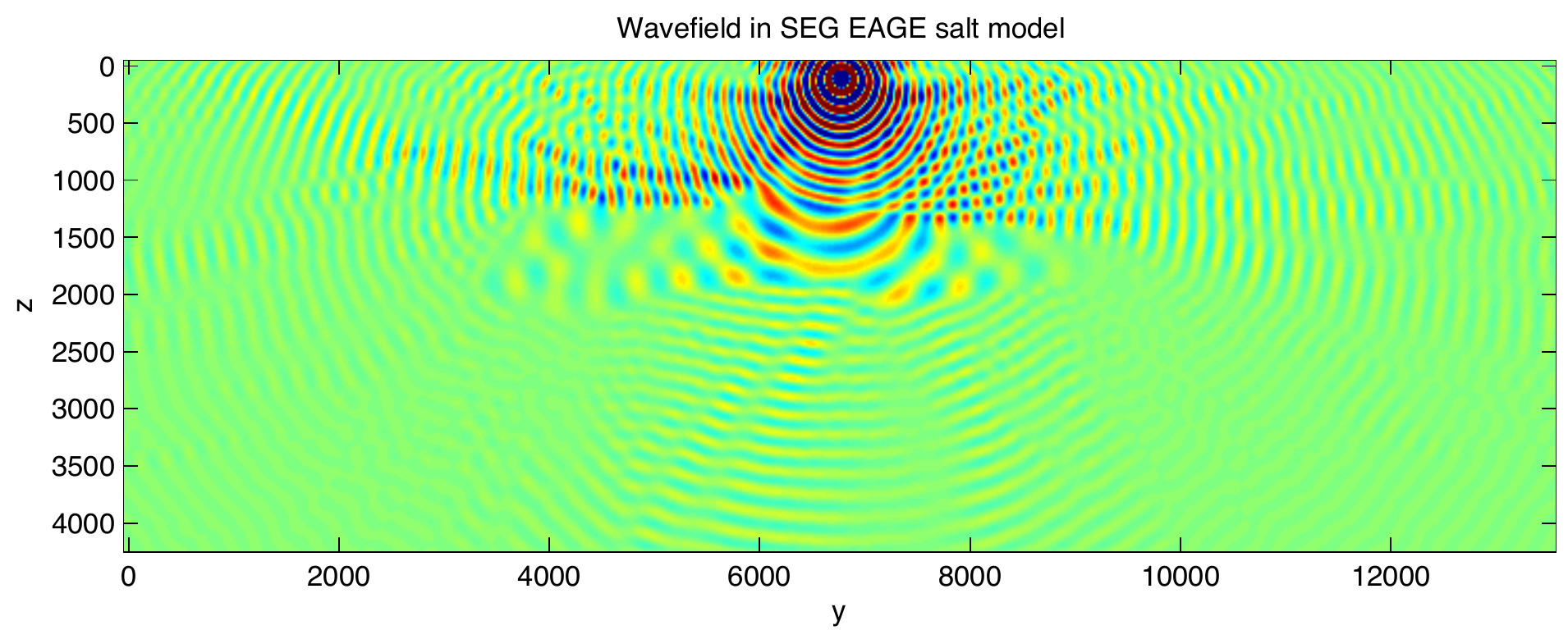}%
\hspace*{-1cm}
\end{center}
\caption{Solution to the Helmholtz equation at $12.5$ Hz: 
(a) $(x,z)$ slice at $y=6740$ m
(b) $(y,z)$ slice at $x=6740$ m.} 
\label{fig:saltmodel_solution}
\end{figure}
\begin{table}
\begin{center}
\begin{tabular}{|l|c|c|c|c|} \hline
frequency 
  & 6.25
  & 7.87
  & 9.91
  & 12.5 
\\
size
  & 338x338x106
  & 426x426x132
  & 536x536x166
  & 676x676x210
\\
\# dof
  & $1.3 \cdot 10^7$
  & $2.5 \cdot 10^7$
  & $5.0 \cdot 10^7$
  & $1.0 \cdot 10^8$
\\
cores
  & 32
  & 64
  & 128
  & 256
\\
\# of rhs.\
  & 1
  & 2 
  & 4
  & 8
\\
iterations
  & 12
  & 12
  & 13
  & 15
\\
computation time/rhs(s)
  & 26
  & 35
  & 45
  & 73
\\ \hline
\end{tabular}
\end{center}
\caption{Computation times and iteration counts for the SEG-EAGE Salt
  Model example.}
\label{tab:SaltModel}
\end{table}

\section{Discussion}
\label{sec:discussion}

Here we summarize some of the conclusions and further discuss the
results.

Using the results presented one can make a case for the use of coarse
meshes using a minimum of five or six points per wavelength in time
harmonic wave simulations in case $k$ is smooth. This idea is not new,
in the exploration geophysics community it appears to be quite
common. However, we found that the methods that have been proposed for
this purpose in \cite{JoShinSuh1996} and \cite{OpertoEtAl2007} can be
expected to give substantial phase errors in simulations of large
distance wave propagation.  By using the new IOFD method (in two or
three dimensions) or the QS-FEM method (in two dimensions only), phase
errors can be made much smaller.

In $k(x)$ has strong gradients, one can expect that, at least locally
where $\nabla k$ is large, finer meshes and/or different
discretizations are needed to obtain accurate solutions. Large
gradients lead to reflections. Typically finer meshes are needed to
model these accurately. For one reason this is because finer
discretizations of $k$ are needed, since linearized scattering theory
shows that reflected waves are associated with perturbations in the
medium velocity with wave vectors of length up to $2k$ (where here $k$
refers to the background velocity around which the linearization is
applied). However, in this case the multigrid approach discussed in
section~\ref{sec:twogrid} can still be useful. It has been applied in
successfully in examples with discontinuities. This suggests to do
further research on multigrid approaches with compact finite
difference method at the coarse level and other discretizations at the
fine level, including methods with local refinement.  A similar
argument can be held for the discretization of the right hand side $f$
in the equation $(-\Delta -k(x)^2)u = f$. For rapidly varying
functions $f$, finer meshes may be needed at least locally where the
rapid variations occur.

When applied in inversion algorithms IOFD and QS-FEM are somewhat more
complicated than the methods of \cite{JoShinSuh1996} and
\cite{OpertoEtAl2007}, because the operatore depends in a more
complicated fashion on the coefficients, which means it is more
complicated to compute the derivative of the finite difference
operator with respect to the medium coefficients. Due to the use of
Hermitian interpolation, these derivatives are however continuous for
our IOFD method.


\bibliographystyle{abbrv}
\bibliography{helmiofd,revisionsbib}

\appendix

\section{Phase slowness computations for finite element methods}

In a periodically repeating setting, which is the case for finite
elements with $N \ge 2$, the situation is somewhat more
complicated. The symmetry property (\ref{eq:translation_symmetry})
only holds for $p,q,r$ divisible by $N$.
For such operators we consider the Bloch waves
\begin{equation} \label{eq:Bloch_wave}
  u_{l,m,n} = e^{i \xi \cdot x_{l,m,n}} v_{l,m,n}
\end{equation}
where $v_{l,m,n}$ is periodic with shifts $(pN,qN,rN)$, $p,q,r$
integers.  For given $\xi$, the action of an operator $A$ with these
symmetries is given by an $N^3 \times N^3$ matrix acting on the
$v_{l,m,n}$ (in three dimensions) for $0\le l,m,n \le N-1$. 
We need to find the vectors $\xi$
for which there is a zero eigenvalue.  However not all zero
eigenvectors correspond to plane waves with wave vectors $\xi$, 
because of the presence of $v_{l,m,n}$. In general $v_{l,m,n}$ can 
correspond to a linear combination of plane waves with wave vectors given by 
$(p \frac{2\pi}{Nh} , q \frac{2\pi}{Nh}, r \frac{2\pi}{Nh})$, where
$p,q,r$ are integers.  Assuming that some $\xi$ corresponds to a
simple zero eigenvalue and that $u_{l,m,n}$ is close to a plane wave
(which is often the case because the eigenfunctions of the continuous
operator are plane waves and the operator $A$ is a good approximation
of the continuous operator), these integers $p,q,r$ can be determined
modulo $N$, and the wave vector associated with an element of the zero 
set of $A$ can be determined.

In the computations we will take a somewhat different approach. We
will compute all eigenvalues, and then only consider the one whose
eigenvector $v_{l,m,n}$ is most closely correlated with (has the in
absolute value largest inner product with) the constant function
$\tilde{v}_{l,m,n} = 1$.  We will say we have found a phase velocity
vector at some $\xi$ if this eigenvalue is zero. This approach has
some limitations, but a more extensive study of this topic falls outside the
scope of this paper. For standard finite elements and $k,h$ such that
the mesh has more than four points per wavelength this appeared to be
sufficient.

\end{document}